\documentclass[cmp]{svjour} 

\usepackage{cite,amssymb,amsfonts,stackrel,color}
\usepackage[numbers]{natbib}
\spnewtheorem{ass}{Assumption}{\bf}{\rm}
\newcommand{\vr}{{\varrho}}
\newcommand{\ve}{{\varepsilon}}
\newcommand{\rmd}{{\rm d}}
\newcommand{\bx}{{ \mathbf{x} }}

\newcommand{\br}{{\mathbf{r}}}

\newcommand{\bq}{{\mathbf{q}}}
\newcommand{\bu}{{ \mathbf{v}}}
\newcommand{\bJ}{{ \mathbf{J}}}
\newcommand{\bv}{{\mathbf{u}}}
\newcommand{\bj}{{\mathbf{j}}}
\newcommand{\bT}{\mathbf{T}}

\newcommand{\bS}{{\mathbf{S}}}

\newcommand{\bI}{{\mathbf{I}}}
\newcommand{\bF}{{\bf F}}
\newcommand{\be}{\begin{equation}}
\newcommand{\ee}{\end{equation}}
\newcommand{\bzed}{{\bf 0}}
\newcommand{\cG}{{\mathcal{G}}}
\newcommand{\cO}{{\mathcal{O}}}
\newcommand{\Dlim}{{\mathcal{D}'}\mbox{-}\lim}
\newcommand{\nj}{{\mbox{\boldmath $\j$}}}

\newcommand{\diver}{{\rm div}_x}
\newcommand{\ol}[1]{\mkern 1.5mu\overline{\mkern-1.5mu#1\mkern-1.5mu}\mkern 1.5mu}
\newcommand{\ul}[1]{\mkern 1.5mu\underline{\mkern-1.5mu#1\mkern-1.5mu}\mkern 1.5mu}
\newcommand{\Dto}{{\stackrel{\mathcal{D}'}{\longrightarrow}}}
\newtheorem{innercustomthm}{Theorem}

\newenvironment{customthm}[1]
  {\renewcommand\theinnercustomthm{#1}\innercustomthm}
  {\endinnercustomthm}
  
\journalname{Communications in Mathematical Physics}

\begin{document}

\title{{An Onsager Singularity Theorem for Turbulent Solutions of Compressible Euler Equations}}
\titlerunning{Onsager Singularity Theorem}
\author{Theodore D. Drivas\inst{1} and Gregory L. Eyink \inst{1}\fnmsep\inst{2}}
\institute{Department of Applied Mathematics \& Statistics\\ The Johns Hopkins University, Baltimore, MD 21218, USA\\ \email{tdrivas2@jhu.edu} \and 
 Department of Physics and Astronomy \\ The Johns Hopkins University, Baltimore, MD 21218, USA\\
 \email{eyink@jhu.edu}}
\authorrunning{T. Drivas \& G. Eyink}

\date{\today}
\communicated{???}

\maketitle
\begin{abstract}
We prove that bounded weak solutions of the compressible Euler equations will conserve thermodynamic entropy 
unless the solution fields have sufficiently low space-time Besov regularity. A quantity measuring kinetic energy cascade 
will also vanish for such Euler solutions, unless the same singularity conditions are satisfied. It is shown furthermore that strong 
limits of solutions of compressible Navier-Stokes equations that are bounded and exhibit anomalous dissipation are 
weak Euler solutions. These inviscid limit solutions have non-negative anomalous entropy production and 
kinetic energy dissipation, with both vanishing when solutions are above the critical degree of Besov regularity.  
 Stationary, planar shocks in Euclidean space with an ideal-gas equation of state provide simple examples 
 that satisfy the conditions of our theorems and which demonstrate sharpness of our $L^3$-based conditions.
These conditions involve space-time Besov regularity, but we show that they are satisfied by 
Euler solutions that possess similar space regularity uniformly in time.
\end{abstract}

%%%%%%%%%%%%%%%%%%%%%%%%%%%%%%%%%%%%%%%%%%%%%%%%%%%%%%%%%%%%%%%%%%%%%%%%%%%%%%%%%%%%%%%%%%%%%%%%%%%%%%%%%%%%%%%%%%%%%%%%%%%%%%%%%%%%%%%%%%%%%%%%%%%%%%%%%%%%%%%%%%%%%%%%%%%%%%%%%%%%%%%%%%%%%%%%%%%%%%%%%%%%%%%%%%%%%%%%%%%%%%%%%%%%%%%%%%%%%%%%%%%%%%%%%%%%%%%%%%%%%%%%%%%%%%%%%%%%%%%%%%%%%%%%%%%%%%

\section{Introduction}

In a 1949 paper on turbulence in incompressible fluids \cite{Onsager49}, L. Onsager announced a result that spatial H\"older exponents $\leq 1/3$ are required of the velocity field for anomalous turbulent dissipation (that is, energy dissipation non-vanishing in the limit of zero viscosity). His sketched argument involved the idea that the velocity field in the limit of infinite Reynolds number is a weak (distributional) solution of the incompressible Euler equations. Onsager never published a detailed proof of his singularity theorem, but works of Eyink \cite{eyink1994energy}, Constantin et al. \cite{CET1994}, and Duchon \& Robert \cite{DuchonRobert2000}, among others later, proved Onsager's claimed result and even more precise results. Onsager's own unpublished argument was essentially the same as that given in \cite{DuchonRobert2000}, according to the historical evidence \cite{EyinkSreenivasan06}. More recent mathematical work has established existence of dissipative weak Euler solutions of the type conjectured by Onsager, beginning with pioneering work of DeLellis \& Sz\'ekelyhidi, Jr. \cite{DeLellisSzekelyhidi10,de2012h} on the convex integration approach, that has since culminated in constructions of solutions with the critical $1/3$ regularity \cite{Buckmaster13,isett2016proof}. None of these theorems establish that dissipative Euler solutions exist as the zero-viscosity limits of incompressible Navier-Stokes solutions, necessary to rigorously found Onsager's theory for fluid turbulence from first principles. 

In this paper, we prove an Onsager singularity theorem for weak solutions of the compressible Euler equations in arbitrary 
space-dimension $d\geq 1.$  The basic state variables are the mass density $\vr:= \vr(\bx,t)$, fluid velocity $\bu:= \bu(\bx,t)$ 
and internal energy density $u:= u(\bx,t)$ (or specific internal energy $u_m=u /\vr$), with the latter defined 
implicitly by the relation $E:= \frac{1}{2}\vr |\bu|^2+u$ in terms of the total energy density $E$. The Euler system then consists of the 
$d+2$ dynamical equations expressing conservation of mass, momentum and energy:
\begin{eqnarray} \label{E-eq:rho}
\partial_t \vr + \nabla_x\!  &\cdot & \!(\vr\bu) = 0 , \\ \label{E-eq:u}
\partial_t (\vr \bu) + \nabla_x \!  &\cdot & \!\left(\vr \bu \bu +p\bI\right)  = 0, \\
\label{E-eq:TE}
\partial_t E+ \nabla_x \!  &\cdot & \!\left((p+E)\bu \right)  =0.
\end{eqnarray}
We use the ``dyadic product" notation $\bu\bu$ of J. W. Gibbs for the tensor product $\bu\otimes\bu $
of space-vectors, which is convenient in this paper.  
The pressure is given by a \emph{thermodynamic equation of state} $p:=p(u,\vr)$ as a function 
of $u$ and $\vr.$ A previous paper \cite{feireisl2016regularity} has studied a similar problem, but 
under the assumption of a barotropic equation of state, with pressure $p=p(\vr)$ a function only of mass 
density and with no independent equation for the total energy density $E.$ Our results are valid for a general 
equation of state $p(u,\vr),$ assuming only that the fluid undergoes no phase transitions during its evolution 
(see Assumption \ref{smoothAss} for a more precise statement).  We also consider strong limits of solutions 
of the compressible Navier-Stokes equations for Reynolds and P\'eclet numbers tending to infinity. As we shall show, 
such strong limits are weak solutions of the compressible Euler system (\ref{E-eq:rho})--(\ref{E-eq:TE}). 
This is a subclass of all Euler solutions, but arguably the one most relevant to compressible fluid turbulence. 

In order to state precisely our results, recall that the Navier-Stokes-Fourier system (or, simply, the compressible 
Navier-Stokes equations) for a viscous, heat-conducting fluid 
takes the form: 
\begin{eqnarray} \label{eq:rho}
\partial_t \vr + \nabla_x   \!&\cdot &\!  (\vr\bu) = 0, \\ \label{eq:u}
\partial_t (\vr \bu) + \nabla_x   \!&\cdot &\!   \left(\vr \bu \bu +p\bI +\bT\right)  = 0 ,\\
\label{eq:TE}
\partial_t E+ \nabla_x   \!&\cdot &\!   \left((p+E)\bu +\bT\cdot \bu + \bq \right)  =0.
\end{eqnarray}
 The \emph{viscous stress tensor}  $\bT$ is given by \emph{Newton's rheological law}:
\be\nonumber
\bT := -2  \eta\bS -\zeta\Theta\bI \ \ \ {\rm with} \ \ \  \bS  := \frac{1}{2}\left(\nabla_x\bu + (\nabla_x\bu)^\top-\frac{2}{d} \Theta\bI\right)\ \ \ {\rm and} \ \ \ \Theta:= \diver\bu,
\ee
where $\eta:= \eta(u,\vr) >0$ and $\zeta:= \zeta(u,\vr)>0$ represent the shear and bulk viscosity, respectively.  
The \emph{heat flux} $\bq$ is given by  \emph{Fourier's law}:
\be\label{eq:q}
\bq := -\kappa\nabla_xT,
\ee
with thermal conductivity $\kappa:=\kappa(u,\vr)>0,$ where $T:=T(u,\vr)$ is the \emph{temperature} of the fluid. For this system, 
see standard physics texts such as Landau \& Lifshitz 
\cite{LandauLifshitz87} (\S 49) or de Groot \& Mazur \cite{deGrootMazur84}, (Ch. XII, \S 1),  and, in the 
mathematics literature, Gallavotti \cite{gallavotti2013foundations} (\S 1.1), Feireisl 
\cite{feireisl2004dynamics,feireisl2013inviscid} or Lions \cite{lions1998mathematical}. 
Balance equations of kinetic energy density and internal energy density follow straightforwardly for smooth 
solutions of the system (\ref{eq:rho})--(\ref{eq:TE}). The equations for kinetic and internal energy densities are:
\begin{eqnarray} \label{eq:KE}
\partial_t \left(\frac{1}{2}\vr |\bu|^2\right) + \nabla_x   \!&\cdot &\!    \left(\left(p+\frac{1}{2}\vr |\bu|^2 \right)\bu +\bT\cdot \bu \right)  = p \ \Theta-Q, \\  \label{eq:IE}
\partial_t  u + \nabla_x  \!&\cdot &\!   \left( u\bu+\bq\right)  = Q  - p\  \Theta, 
\end{eqnarray}
where the rate of viscous heating of the fluid is explicitly:
\be\label{eq:Q}
Q :=-\bT: \nabla_x\bu=2\eta|\bS|^2 +\zeta \Theta^2.
\ee

An essential role will be played in our analysis by the {\it thermodynamic entropy}.
The entropy density $s:=s(u,\vr)$ (or the specific entropy $s_m=s/\vr$) is related to 
$u$ and $\vr$ through the first law of thermodynamics in the form:
\be\label{eq:Gibbs}
T \rmd s= \rmd u- \mu  \rmd  \vr, 
\ee
 with the \emph{chemical potential} $\mu:=\mu(u,\vr)$.  The entropy $s$ is a concave function of $(u,\vr),$  
 as a consequence of extensivity of the thermodynamic limit \cite{martin1979statistical,ruelle1999statistical}
 or macroscopically as an expression of thermodynamic stability \cite{callen1985thermodynamics,evans2015entropy}. 
 The {\it fundamental equation} $s:=s(u,\vr)$ completely determines the thermodynamics of any system, 
 yielding by equilibrium thermodynamic relations all other functions, including temperature $T(u,\vr),$  
 chemical potential $\mu(u,\vr),$ pressure $p(u,\vr),$ etc. These functions satisfy the thermodynamic \emph{Gibbs relation}:
 \be\label{eq:homogenousGibbs}
T s = u+p-\mu\vr, 
\ee
by an application of the Euler theorem on homogeneous functions \cite{callen1985thermodynamics,evans2015entropy}.
\begin{remark} \label{EOSremark}
For concreteness, we mention here a couple of examples of thermodynamic fundamental equations of some standard 
fluids. First, an \emph{ideal gas} has 
\begin{eqnarray} \label{idealGasEOS}
s(u,\vr) =  \alpha k_B \vr\left[\log\left(\frac{u}{\vr^{1+1/\alpha}}\right)+s_0\right] 
\end{eqnarray}
for Boltzmann's constant $k_B$ and parameter $\alpha=f/2>0,$ related to the number of mechanical degrees of freedom
$f$ of individual gas molecules. For a simple monatomic gas in $d$ space dimensions, $f=d.$ The constant $s_0$ is 
determined from microscopic statistical mechanics. This simple model with an appropriate choice of $\alpha$ describes 
the thermodynamics of most gaseous systems at low density.

Another standard example is the \emph{van der Waals fluid} with entropy: 
\begin{eqnarray} \label{vanderWaalsEOS}
s(u,\vr) =  {\rm conc.\ env.}\left\{\alpha k_B \vr \left[\log\left((1/\vr-b)^{1/\alpha} (u/\vr+a\vr)\right)+s_0\right]\right\},
\end{eqnarray}
Here the notation ``conc.\ env.'' denotes the upper concave envelope of the function inside the curly brackets, which is smooth but not 
a globally concave function of $(u,\vr).$  The van der Waals model incorporates some density corrections through the new 
terms involving constants $a,b>0,$ but reduces to the ideal gas law in the low-density limit $\rho\to 0.$ This is the simplest 
example of a fluid model exhibiting a gas-liquid phase transition for low energies and high densities, at the points 
in the $(u,\vr)$-plane of non-smoothness of the concave envelope in (\ref{vanderWaalsEOS}).

For these models, see  \cite{callen1985thermodynamics,evans2015entropy}. Needless to say,
our results apply not just to these specific examples but very widely, because the relations
(\ref{eq:Gibbs}) and (\ref{eq:homogenousGibbs}) are general results of equilibrium thermodynamics 
and statistical mechanics \cite{martin1979statistical,ruelle1999statistical}.
\end{remark}
From the compressible Navier-Stokes system (\ref{eq:rho})--(\ref{eq:TE}) and the thermodynamic relation 
(\ref{eq:Gibbs}) follows the balance equation for the entropy density:
\be\label{eq:s}
\partial_t s+ \nabla_x \cdot \left(s\bu + \frac{\bq}{T} \right)  =   \frac{Q }{T}  + \Sigma_\kappa.
\ee
The entropy production rate $\Sigma := {Q }/{T}  + \Sigma_\kappa$ involves a viscous heating contribution with $Q$ again given by (\ref{eq:Q}), and a term due to thermal conduction:
\be\label{eq:Lamb}
\Sigma_\kappa := -\frac{\bq\cdot \nabla_xT}{T^2}= \kappa \frac{|\nabla_x T|^2}{T^2}.
\ee
In accord with second law of thermodynamics, entropy is globally increased since:
\be\label{dissterms}
\Sigma := \frac{Q}{T} +\Sigma_\kappa = 2\frac{\eta}{T}|\bS|^2 +\frac{\zeta}{T} |\Theta|^2 +  \kappa \frac{|\nabla_x T|^2}{T^2}\geq 0.
\ee
For these standard results see \cite{LandauLifshitz87,deGrootMazur84}. 

Smooth solutions of the compressible Euler system satisfy the same balance equations as (\ref{eq:KE}), (\ref{eq:IE}), 
and (\ref{eq:s}), but with $\zeta,\eta,\kappa\equiv 0$ so all of the non-ideal terms vanish, i.e.  $\bT,\bq=0$ and $Q,\Sigma\equiv 0$. This 
need not be true, of course, for weak solutions. 
An important class of weak solutions that we consider are those arising from limits of solutions $\vr^{\varepsilon},u^{\varepsilon},\bu^{\varepsilon}$
of the Navier-Stokes system with transport coefficients scaled as $\eta^\ve=\ve\eta,$  $\zeta^\ve=\varepsilon\zeta,$ 
$\kappa^\ve=\varepsilon\kappa,$ for $\varepsilon \rightarrow 0.$ Essentially, $1/\varepsilon$ represents the Reynolds 
and P\'eclet numbers of the fluid. To avoid issues involving boundary conditions, we consider only flows on space
domains $\Omega$ either $d$-dimensional Euclidean space $\Omega=\mathbb{R}^d$ or the $d$-torus 
$\Omega=\mathbb{T}^d$. We shall often use the notation $\Gamma =  \Omega\times (0,T)$ for the 
space-time domain, $T<\infty$ or $T=\infty.$
  
 We then make the following specific assumptions: 

\begin{ass}\label{assum1}
 Given $\varepsilon>0$, we assume that there exists a unique smooth solution $u^\varepsilon,\vr^{\varepsilon},\bu^\varepsilon$ of the compressible 
 Navier-Stokes system (\ref{eq:rho})--(\ref{eq:TE}) on $\Omega \times (0,T)$ for a given equation of state.  In fact, most of our analysis 
will apply to suitable weak Navier-Stokes solutions. We assume $u^\varepsilon,\vr^\varepsilon,\bu^\varepsilon\in L^\infty(\Omega \times (0,T))$ 
uniformly bounded for $\ve<\ve_0$ and that for some $1\leq p<\infty$
strong limits exist 
\be
u^{\varepsilon}\to u, \ \  \vr^{\varepsilon}\to \vr,\ \  \bu^{\varepsilon}\to \bu \ \ {\rm in} \  \ L_{loc}^p(\Omega \times (0,T)).
\label{ass-str-conv} 
\ee
Here $L_{loc}^p(\Gamma)$, as usual (see e.g. \cite{gilbarg2015elliptic,evans2015measure})
, denotes the linear space of measurable functions which are locally $p$-integrable:
\be
L_{loc}^p(\Gamma) = \{ f:\Gamma\to \mathbb{R} \ {\rm meas.} \ | \ f\in L^p(O),\  \forall\ {\rm open}\  O\subset\subset \Gamma\}
\ee
 where $A\subset\subset B$ denotes that the closure $\bar{A}$ is compact and $\bar{A}\subset B$. Strong convergence 
$f_n\to f$ in $L_{loc}^p(\Gamma)$ is the requirement that for any open $O\subset\subset \Gamma$ the restrictions 
converge $f_n\big|_O\to f\big|_O$ strong in $L^p(O).$  With this topology,  $L_{loc}^p(\Gamma)$ 
is a complete metrizable space for all $p\geq 1$.  Whenever $\bar{\Gamma}$ is itself compact 
(e.g. $\bar{\Gamma} = \mathbb{T}^d\times [0,T]$ with $T<\infty$), $L_{loc}^p(\Gamma)= L^p(\Gamma)$.  
We remark also that, trivially,  $L^\infty(\Gamma)\subset L_{loc}^p(\Gamma)$ for all $p\geq 1$. 
Thus the convergence in (\ref{ass-str-conv}) implies convergence pointwise almost everywhere for a subsequence $\ve_k\rightarrow 0$ and 
$u,\vr,\bu\in L^\infty(\Omega \times (0,T)).$ 
The mode of convergence (\ref{ass-str-conv}) permits limiting fields with jump discontinuities. We also assume $\vr^\ve\geq \vr_0$ 
for some $\vr_0>0$ and $\varepsilon<\varepsilon_0, $ so that the fluid nowhere approaches a vacuum state with zero density. 
\end{ass}

\begin{ass}\label{smoothAss}
We assume that the solutions involve 
thermodynamic states $(u,\vr)$ strictly away from  phase transitions, so that all thermodynamic functions 
$h=p,$ $T,$ $\mu,$ $s$, $\eta,$ $\zeta,$ $\kappa,$ etc. are smooth in $u,$ $\vr$. The set of states 
attained by any solution is the \emph{essential range} over space-time, ${\mathcal R}={\rm ess.ran}(u,\vr)$ and 
${\mathcal R}^\ve={\rm ess.ran}(u^\ve,\vr^\ve)$
for $\ve>0,$ which are compact sets in $\mathbb{R}^2$ \cite{rudin1987real}. 
The uniform boundedness in $L^\infty(\Omega \times (0,T))$ 
of $u^\ve,$ $\vr^\ve$ for $\ve<\ve_0$ implies that there exists a compact set $K\subset \mathbb{R}^2$ such that the 
closed convex hull 
\be {\rm conv}[{\mathcal R}^\ve \cup {\mathcal R}]\subseteq K, \quad \forall \ve<\ve_0. \ee
We then assume for $h$ that there is an open set $U\subset \mathbb{R}^2,$ with $K\subset U$ and $h\in C^M(U)$
with smoothness exponent $M\geq 2.$
\end{ass}

\begin{ass}\label{assum1b}  
Assume that the dissipation terms defined in equations (\ref{eq:Q}) and (\ref{dissterms}) 
converge as $\varepsilon\rightarrow 0$ in the sense of distributions:
$$
 Q_\eta^{\varepsilon}:= 2\eta^\ve |\bS^\ve|^2, \quad Q_\zeta^{\varepsilon}:= \zeta^\ve (\Theta^\ve)^2, \quad Q^\ve:=Q^\ve_\eta+ Q^\ve_\zeta\ \Dto\  Q,
$$
and 
$$
 \Sigma_\eta^{\varepsilon}:= \frac{Q_\eta^{\varepsilon}}{T^\ve}, \quad 
 \Sigma_\zeta^{\varepsilon}:= \frac{Q_\zeta^{\varepsilon}}{T^\ve}, \quad 
 \Sigma_\kappa^\varepsilon := \kappa^\ve\left| \frac{\nabla_x T^\varepsilon}{T^\varepsilon}\right|^2, \quad 
\Sigma^\ve:= \Sigma_\eta^\ve+\Sigma_\zeta^\ve+\Sigma_\kappa^\ve \ \Dto \ \Sigma.
 $$
The limit distributions are obviously non-negative, and thus Radon measures. 
 \end{ass}

\begin{remark} \label{rem:NSshockSol}
The set of compressible Navier-Stokes solutions on Euclidean space $\mathbb{R}^d$ 
satisfying these three assumptions is non-empty and includes, 
in particular, shock solutions. See examples in \cite{johnson2014closed} and \cite{eyinkdrivascompr}.  
%(albeit in Euclidean space $\mathbb{R}^d$ rather than the torus $\mathbb{T}^d.$)
Numerical simulations of compressible turbulence with the system (\ref{eq:rho})--(\ref{eq:TE}) 
on the torus $\mathbb{T}^d$ show that small-scale shocks (or ``shocklets'')
naturally develop. There is also some evidence, however, that at sufficiently high Mach numbers the limiting mass density 
$\vr$ as $\varepsilon\rightarrow 0$ may exist only as a measure and not as a bounded function \cite{kim2005density}. 
There is thus empirical motivation to weaken Assumption \ref{assum1} in future work. 
\end{remark}

We now state our main theorems. First, we establish the balance equations of energy and entropy for general bounded weak Euler solutions :

\begin{theorem}\label{Onsager1E}
\noindent  Let $u,\vr,\bu\in L^\infty(\Omega \times (0,T))$ be any weak solution of the compressible Euler system (\ref{E-eq:rho})--(\ref{E-eq:TE}) satisfying $\vr\geq\vr_0>0$ and Assumption \ref{smoothAss}.  Let $Q_\ell^{\rm flux}$ be the ``energy flux'' defined by 
(\ref{eq:Quell}) below and $\Sigma_\ell^{\rm inert *}$ the ``inertial entropy production'' defined by (\ref{eq:SigInert}). Assuming that 
the distributional limit of $Q^{\rm flux}_\ell$ exists, 
\be\label{QcascadeDefects}
 Q_{\rm flux}= \Dlim_{\ell\to 0} Q_\ell^{\rm flux}
 \ee
then local energy and entropy balance equations hold in the sense of distributions on $\Omega \times (0,T)$:
 \begin{eqnarray}
 \label{eq:KE*Euler}
\partial_t \left(\frac{1}{2}\vr |\bu|^2\right) + \nabla_x \cdot \left(\left(p+\frac{1}{2}\vr |\bu|^2\right) \bu \right)   &=& p \circ\Theta-  Q_{\rm flux},\\  
\label{eq:IE*Euler}
\partial_t  u + \nabla_x \cdot \left( u\bu\right)   &=&  Q_{\rm flux}- p \circ \Theta,\\
 \label{eq:s*Euler}
\partial_t s+ \nabla_x \cdot \left(s\bu \right)  &=&   \Sigma_{\rm inert}.
\end{eqnarray}
where  $\Sigma_{\rm inert}$ and $p \circ\Theta$ necessarily exist and are defined by the distributional limits
\be\label{cascadeDefects}
\Sigma_{\rm inert}=\Dlim_{\ell\to 0} \Sigma_\ell^{\rm inert *}, \quad \quad  p\circ\Theta = \Dlim_{\ell\rightarrow 0} (p*\cG_\ell)(\Theta*\cG_\ell), \ee
with $\cG_\ell,$ $\ell>0$ a space-time mollifying sequence. 
 \end{theorem}

%\vspace{-6mm}
\begin{remark} \label{DuchonRobertRem1}
This result is analogous to Proposition 2 of \cite{DuchonRobert2000} for weak solutions of incompressible Euler
with $\bu\in L^3(\mathbb{T}^d \times (0,T))$. In their theorem, the assumption on the existence of $Q_{\rm flux}$ 
was unnecessary. We need to add this as an additional hypothesis,
because of the new term $p\circ\Theta$ that appears in the energy balance equations. Of course, $p\circ\Theta=0$ 
assuming incompressibility.
 \end{remark} 

\begin{remark}
Note that the second equation in (\ref{cascadeDefects}) for $p\circ\Theta$ is a standard definition of a generalized distributional 
product of $p$ and $\Theta$ \cite{Oberguggenberger92}. This standard definition requires that the limit be independent of the
chosen mollifier $\cG$.  We note that for the purposes of Theorem \ref{Onsager1E}, one could alternatively assume existence 
of $p \circ\Theta$ and then deduce it for $Q_{\rm flux}$. The combination $p\circ \Theta-Q_{\rm flux}$ always exists. 
 \end{remark} 

Our next results concern the strong limits of Navier-Stokes solutions satisfying Assumptions \ref{assum1} -- \ref{assum1b}.  
 First, we prove that these limits are necessarily weak solutions of the Euler equations, even if the limit dissipation measures 
 in Assumption \ref{assum1b} remain positive: $Q>0$ and $\Sigma>0.$  Moreover, we show that such solutions satisfy weak energy and entropy balance laws which include possible anomalies: 
 
\begin{theorem}\label{Onsager1NS}
\noindent  The strong limits $u,\vr,\bu$ of compressible Navier-Stokes solutions under Assumptions \ref{assum1} -- \ref{assum1b} are weak solutions of the compressible Euler system (\ref{E-eq:rho})--(\ref{E-eq:TE}) on $\Omega \times (0,T)$. 
Furthermore, the following local energy and entropy equations hold in the sense of distributions on $\Omega \times (0,T)$:
 \begin{eqnarray}
 \label{eq:KE*}
\partial_t \left(\frac{1}{2}\vr |\bu|^2\right) + \nabla_x \cdot \left(\left(p+\frac{1}{2}\vr |\bu|^2\right) \bu \right)   &=& p *\Theta-  Q,\\  
\label{eq:IE*}
\partial_t  u + \nabla_x \cdot \left( u\bu\right)   &=&  Q- p *\Theta,\\
 \label{eq:s*}
\partial_t s+ \nabla_x \cdot \left(s\bu \right)  &=&   \Sigma,
\end{eqnarray}
with $Q\geq 0$ and $\Sigma\geq 0$ given by Assumption \ref{assum1b} and with 
\be p*\Theta := \Dlim_{\ve\to 0}\  p^\ve \Theta^\ve, \label{pDil} \ee
where this distributional limit necessarily exists. 
\end{theorem}

%Analogous to Remark \ref{DuchonRobertRem1} for Theorem \ref{Onsager1E}, we note:
\begin{remark}
Theorem \ref{Onsager1NS} is analogous to Proposition 4 of \cite{DuchonRobert2000} for the strong limits of solutions
of the incompressible Navier-Stokes equation with viscosity tending to zero. Again, in their theorem, the analogue 
of our Assumption \ref{assum1b} was unnecessary, whereas we needed to add this as an additional hypothesis because of the new term $p*\Theta$ defined by (\ref{pDil}) that appears in the energy balance equations. 
\end{remark}

%\red{\begin{remark} 
%If we restrict our considerations to barotropic compressible models in which the pressure $p=p(\vr)$ is taken to be a function of mass density only, we note that $p*\Theta$ is well-defined as a distribution without the need for Assumption \ref{assum1b}.  This follows since, in these models, the internal energy per unit mass is  
%$$
%e(\vr)= \int p(\vr) \frac{d\vr}{\vr^2}.
%$$
%The non-ideal evolution equation for $u^\varepsilon=\vr^\varepsilon e^\varepsilon$ is then found to be:
%$$
%\partial_t  u^\varepsilon + \nabla_x  \cdot  \left( u^\varepsilon\bu^\varepsilon\right)  =  - p^\varepsilon \Theta^\varepsilon.
%$$
%Notably, no dissipative terms enter into the internal energy evolution equation, which implies that total-energy cannot be conserved, a pathology of all (dissipative) barotropic models.   By the arguments in the proof of Theorem \ref{Onsager1NS}, the strong convergence assumed for $\vr^\varepsilon,u^\varepsilon\to \vr, u$ implies that $p*\Theta = \partial_t  u + \nabla_x  \cdot  \left( u\bu\right)$ exists as a distribution.  
%\end{remark}}

\begin{remark}
Euler solutions obtained from Theorem \ref{Onsager1NS} for vanishing viscosity necessarily satisfy Theorem \ref{Onsager1E} 
for general weak Euler solutions.  It follows that:
\be\label{compressible45ths}
 \Sigma_{\rm inert} = \Sigma \geq 0  \quad\quad {\rm and } \quad\quad   Q_{inert}:=Q_{\rm flux} + \tau(p,\Theta)=Q\geq 0,
\ee
where $\tau(p,\Theta)$ is the ``pressure-dilatation defect'' defined by
 \be \tau(p,\Theta)= p*\Theta-p\circ \Theta \label{pressDildef} .\ee
The lefthand sides in (\ref{compressible45ths}) are ``inertial-range'' expressions for $Q$ and $\Sigma$, 
analogous to those established in Proposition 1 and Section 5 of \cite{DuchonRobert2000} for incompressible fluids.   
In particular, $\Sigma_{\rm inert}$ and $Q_{\rm flux}$ describe ``cascade''  and can be expressed in terms of increments 
of the variables $u,$ $\vr,$ $\bu$ by analogues of the Kolmogorov ``4/5th-law'' for compressible turbulence.  
Whereas $\Sigma_{\rm inert},$ $Q_{\rm flux}$ can have any signs for general weak Euler solutions, they are 
constrained by (\ref{compressible45ths}) for zero-viscosity solutions. The pressure-dilation defect %$\tau(p,\Theta)$
in (\ref{pressDildef}) is an additional source of anomalous energy dissipation, with no analogue 
for incompressible fluids.   
\end{remark}

\begin{remark} 
Shock solutions on Euclidean space $\mathbb{R}^d$, as discussed in  \cite{johnson2014closed} and \cite{eyinkdrivascompr}, provide examples  
for which $Q>0$ and $\Sigma>0$ in (\ref{eq:KE*})--(\ref{eq:s*}). It is of some interest to note that for stationary, planar 
shocks in an ideal gas, $Q=\tau(p,\Theta)>0,$ so that the entire contribution to $Q$ is from the 
pressure-dilatation defect. See \cite{eyinkdrivascompr} for this result.   
Although shock solutions with discontinuous state variables $u,$ $\vr,$ $\bu$ provide the simplest examples
of weak Euler solutions with $Q,$ $\Sigma$ positive, presumably positive anomalies can occur even with continuous solutions. 
\end{remark}

We now state an analogue of the Onsager singularity theorem. We prove necessary conditions for anomalous
dissipation involving Besov space exponents, as in the improvement by \cite{CET1994} of Onsager's H\"older-space statement. 
Here we note that the Besov space $B_p^{\sigma,\infty}(O)$ for a general 
open set $O \subset\subset\Gamma$ is made up of measurable functions $f:\Gamma\to \mathbb{R}$ which are finite in the norm:
\be
\| f\|_{B_p^{\sigma,\infty}(O)}:= \| f\|_{L^p(O)} 
+ \sup_{h\in \mathbb{R}^D,|h|<h_O} \frac{\|f(\cdot+h) - f\|_{L^p(O)}}{|h|^\sigma},
\label{besov-def} \ee
for $p\geq 1$ and $\sigma\in (0,1)$ and where $h_O={\rm dist}(O, \partial\Gamma)$.  See \cite{feireisl2016regularity} and, for a general discussion, \cite{triebel2006theory}, \S 1.11.9. 
In this paper, we define a local Besov space:
\be \label{local-besov-def}
B_{p,loc}^{\sigma,\infty}(\Gamma):= \{ f:\Gamma \to \mathbb{R} \ {\rm meas.} \  | \  f\in B_p^{\sigma,\infty}(O), \   \  \forall\ {\rm open}\  O\subset\subset \Gamma\}.
\ee  
Again, whenever $\bar{\Gamma}$ is itself compact (e.g. $\bar{\Gamma }= \mathbb{T}^d \times [0,T]$), 
$B_{p,loc}^{\sigma,\infty}(\Gamma)=B_{p}^{\sigma,\infty}(\Gamma)$.

\begin{theorem}\label{Onsager2}
Let $u,\vr,\bu\in L^\infty(\Omega \times (0,T))$ be any weak solution of the compressible Euler system (\ref{E-eq:rho})--(\ref{E-eq:TE}) satisfying $\vr\geq\vr_0>0$, Assumption \ref{smoothAss}, and additionally
$$
u \in  B_{p,loc}^{\sigma_p^u,\infty}(\Omega\times (0,T)), \ \ \vr \in   B_{p,loc}^{\sigma_p^\vr,\infty}(\Omega\times(0,T)),  \ \ 
\bu \in  B_{p,loc}^{\sigma_p^v,\infty} (\Omega\times (0,T)),
$$
with all three of the following conditions satisfied 
 \begin{eqnarray}
 \label{sigineq1}
 2\min\{\sigma_p^u,\sigma_p^\vr\} + \sigma_p^v&>& 1,\\
 \label{sigineq2}
 \min\{ \sigma_p^u,\sigma_p^\vr\} +2\sigma_p^v&>&1, \\
 \label{sigineq3} 
 3\sigma_p^v&>&1 ,
 \label{regass}
 \end{eqnarray} 
for some $p\geq 3.$ Then $Q_{\rm flux}$, $\Sigma_{\rm flux} $ necessarily exist and equal zero.  Further, inviscid 
limit solutions  from Theorem \ref{Onsager1NS}  satisfying exponent conditions (\ref{sigineq1})-(\ref{sigineq3}) 
have 
$$
Q = \Sigma=0 \quad {\rm and } \quad p *\Theta=p \circ \Theta.
$$
Thus,  it is only possible that $Q>0$ or $\Sigma>0$ if at least one of (\ref{sigineq1})--(\ref{sigineq3}) fails to hold
for each $p\geq 3.$
\end{theorem}

\begin{remark}
Our proof of Theorem \ref{Onsager2} generalizes the argument of \cite{CET1994}, which employed a simple 
mollification of the weak Euler solution. In fact, this idea can be exploited to give a new notion of 
``coarse-grained Euler solution'' which we introduce in section \ref{sec:CGWS} and show there to be equivalent 
to the standard notion of ``weak solution,'' not only for compressible Euler equations but for very general 
balance relations. As discussed in \cite{eyinkdrivascompr}, the concept of ``coarse-grained solution'' 
makes connection with renormalization-group methods in physics.  We employ this notion to prove both our 
Theorems \ref{Onsager1NS} and \ref{Onsager2}. Our analysis of compressible Navier-Stokes and Euler 
solutions was directly motivated by the earlier work of Aluie \cite{aluie2013scale}, and 
our theorems generalize previous results for barotropic compressible flow \cite{feireisl2016regularity}. 
%\red{ Similar results have been proved for variable density incompressible fluids  \cite{leslie2016energy}. } 
It is worth noting that all of our results generalize to relativistic Euler equations in Minkowski spacetime, 
following the discussion in \cite{eyinkdrivasrelat}.  
\end{remark}

\begin{remark}
Our Theorem \ref{Onsager2} is formulated in terms of space-time regularity, whereas the original statement 
of Onsager and most following works have given necessary conditions for anomalous dissipation in terms of 
space-regularity only. Note that our proof of Theorem \ref{Onsager2} requires mollification/coarse-graining in time 
as well as space, and thus space-time regularity is natural for the proof (and also in the relativistic setting). However, 
we  obtain conditions involving space-regularity only from the next theorem. 
Adapting standard definitions, we set:
\begin{eqnarray}  \label{Linfty_into_BesovLoc}
L^\infty((0,T); B_{p,loc}^{s,\infty}(\Omega))&:=&\{f:\Gamma\to \mathbb{R} \ {\rm meas.} \ | \\\nonumber
&& \quad\ \  \sup_{t\in (0,T)}\|f (\cdot,t)\|_{B_{p}^{s,\infty}(O)}<\infty ,\ \  \forall \ {\rm open}\  O\subset\subset \Omega \}.
\end{eqnarray}
With this convention, we have the following result:
\end{remark} 

\begin{theorem}\label{timeReg}
 Let $u,\vr,\bu$ be any weak Euler solution satisfying $\vr\geq\vr_0>0$ and  $\vr,u,\bu\in L^\infty(\Omega \times (0,T))$ together with: 
 $$
 u \in   L^\infty((0,T); B_{p,loc}^{\sigma_p^u,\infty}(\Omega)), \  \vr \in  L^\infty((0,T); B_{p,loc}^{\sigma_p^\vr,\infty}(\Omega)), \ \bu \in  L^\infty((0,T); B_{p,loc}^{\sigma_p^v,\infty} (\Omega)),
$$
for Besov exponents $0\leq \sigma_p^u, \sigma_q^\vr, \sigma_q^v\leq 1$. Then the solutions are also Besov regular 
locally in space-time:
  \begin{eqnarray}
  \label{thm3-2}
   u &\in& B_{p,loc}^{\min\{\sigma_p^\vr, \sigma_p^v,\sigma_p^u\},\infty}(\Omega \times (0,T)),\\
\label{thm3-1}  
\vr &\in& B_{p,loc}^{\min\{\sigma_p^\vr, \sigma_p^v\},\infty}( \Omega \times (0,T)),  \\
\label{thm3-3}    
     \bu &\in& B_{p,loc}^{\min\{\sigma_p^\vr, \sigma_p^v,\sigma_p^u\},\infty}(\Omega \times (0,T)). 
\end{eqnarray}
\end{theorem}

\begin{remark}
This result is very similar to that obtained in recent work of P. Isett for H\"{o}lder-continuous weak solutions of incompressible 
Euler \cite{isett2013}, and the proof is almost the same. In fact, we shall derive Theorem \ref{timeReg} as a consequence 
of a more general result which derives time-regularity from space-regularity for a wide class of weak balance equations. 
\end{remark} 

\begin{remark}
It is interesting to know how sharp are the necessary conditions for anomalous dissipation following from 
Theorems \ref{Onsager2} and \ref{timeReg}. 
While answering this question for the incompressible case has required more sophisticated tools 
\cite{isett2012,isett2014,Buckmaster13,DeLellisSzekelyhidi10}, we have a very cheap argument showing that 
our conditions are sharp for $p=3$ and $\Omega=\mathbb{R}^d$.  In fact, the stationary planar shock
solutions for an ideal gas in \cite{johnson2014closed,eyinkdrivascompr} are obtained as strong limits of compressible 
Navier-Stokes solutions for vanishing viscosity and satisfy $u,\vr,\bu\in (BV_{loc}\cap L^\infty)(\mathbb{R}^d).$ 
These provide a simple example of dissipative Euler solutions saturating our bounds, since $(BV_{loc}\cap L^\infty)(\Omega)\subset B_{p,loc}^{1/p,\infty}(\Omega),$ 
$p\geq 1$ 
by the argument of \cite{feireisl2016regularity}, Proposition 2.1. That paper stated this result only for $\Omega=\mathbb{T}^d,$
but the proof rests on a standard approximation theorem for $BV$ functions that holds for any open $O\subset \mathbb{R}^d$ 
(see e.g. \cite{evans2015measure}, Thm. 2  of \S 5.2.2, or \cite{Ziemer1989}, Thm. 5.3.3). For $p=3$ this means that we may 
take $\sigma_3^u=\sigma_3^\vr=\sigma_3^v=1/3$ and then (\ref{sigineq1})--(\ref{sigineq3}) are satisfied as equalities.  
For $p>3$, the sharpness of our results for solutions on $\mathbb{R}^d$ remains an open issue. Note that 
a standard Besov embedding gives $B^{\sigma,\infty}_{p,loc}(\Omega)\subset C^{\sigma-d/p}_{loc}(\Omega)$ and 
$B^{\sigma,\infty}_{p,loc}(\Omega\times (0,T))\subset C^{\sigma-(d+1)/p}_{loc}(\Omega\times (0,T))$ 
(see \cite{triebel2006theory}, \S 1.11.1). Thus, if our necessary conditions are sharp, 
then dissipative solutions at the critical values for sufficiently large $p$ must be H\"older-continuous.

%Imposing a finite entropy $S=\int d^dx\, s$ and bounded velocities, 
%no stationary Euler solution can illustrate the sharpness of our results.
No stationary Euler solution can illustrate the sharpness of our results, if 
a finite entropy $S=\int d^dx\, s$ and bounded velocities are required. 
If $(1\wedge |\bx|^{-1})s\,\bu \in L^1(\mathbb{R}^d),$ then $\nabla_x\cdot (s\bu)=\Sigma\geq 0$ only for 
$\Sigma\equiv 0.$ This follows by smearing the stationary entropy balance with $\phi(|\bx|/R)$ for $\phi\in C^\infty_c({\mathbb R}^+,{\mathbb R}^+)$
with $\phi(r)=1$ for $r<1,$ $\phi(r)=0$ for $r>2,$ so $\int d^dx\, \Sigma=\lim_{R\to\infty} -\int_{R<|\bx|<2R} d^dx\,
\frac{1}{R}\phi'\left(\frac{|\bx|}{R}\right) s v_r,$ with $v_r$ the radial component of $\bu.$ Thus, $\int d^dx\, \Sigma=0$ with the integrability  
assumption on $s\,\bu,$ e.g. for $\bu\in L^\infty(\mathbb{R}^d)$ and $s\in L^1(\mathbb{R}^d).$ The 
sharpness of our conditions thus remains open for all $p\geq 3$ with such solutions on $\mathbb{R}^d.$ Likewise, 
the question remains open  for Euler solutions on $\mathbb{T}^d$.  No stationary shock examples of the type  
discussed in \cite{johnson2014closed,eyinkdrivascompr} can exist on the torus, since the anomalous entropy production 
in a stationary solution must arise from positivity of the space-divergence of the entropy current, which necessarily vanishes for periodic 
solutions. (We owe both of the above observations to an anonymous referee).  On the other hand, turbulent solutions of the compressible 
Navier-Stokes equation observed in numerical simulations on the torus appear to exhibit non-stationary shocks (e.g. \cite{kim2005density}).
We therefore expect that such shock solutions again illustrate sharpness of our results for $p=3$ and $\Omega=
\mathbb{R}^d$ or $\mathbb{T}^d,$
but the rigorous mathematical construction of such non-stationary solutions will be more involved.
\end{remark}

%\begin{remark}
%Global balances for energy and entropy follow from our Theorem \ref{Onsager}, along with weak continuity of $\frac{1}{2}\vr |\bu|^2$, $u$ and $s$ are implied by the equations (which follows from Lemma 8 of CITE) and our $L^\infty$ assumptions. Integrating our coarse-grained global balances for the viscosity solutions $\vr,e,\bu$ from $s$ to $t$:
%\be
%\left\langle\frac{1}{2} \ol{\vr}_\ell |\tilde{\bu}_\ell |^2(\cdot, t)\right\rangle_\bx +\int_s^t \rmd \tau\left[ \left\langle Q_\ell^{\rm flux}(\cdot, \tau)\right\rangle_\bx - \left\langle \ol{p}_\ell \diver\ol{\bu}_\ell(\cdot,\tau) \right\rangle_\bx \right] = \left\langle\frac{1}{2} \ol{\vr}_\ell |\tilde{\bu}_\ell |^2(\cdot, s^+)\right\rangle_\bx
%\ee
%\be\nonumber
%KE(t) = \int_{\mathbb{T}^d} \!\!\! \rmd \bx \ \vr |\bu|^2 (\bx,t), \ \ \  IE(t) = \int_{\mathbb{T}^d} \!\!\! \rmd \bx \ e(\bx,t), \ \ \  S(t) = \int_{\mathbb{T}^d} \!\!\! \rmd \bx \ s(\bx,t)
%\ee
%\end{remark}

%%%%%%%%%%%%%%%%%%%%%%%%%%%%%%%%%%%%%%%%%%%%%%%%%%%%%%%%%%%%%%%%%%%%%%%%%%%%%%%%%%%%%%%%%%%%%%%%%%%%%%%%%%%%%%%%%%%%%%%%%%%%%%%%%%%%%%%%%%%%%%%%%%%%%%%%%%%%%%%%%%%%%%%%%%%%%%%%%%%%%%%%%%%%%%%%%%%%%%%%%%%%%%%%%%%%%%%%%%%%%%%%%%%%%%%%%%%%%%%%%%%%%%%%%%%%%%%%%%%%%%%%%%%%%%%%%%%%%%%%%%%%%

%\vspace{-10pt}
The detailed contents of the present paper are as follows: 
In section \ref{sec:CGWS} we introduce the space-time coarse-graining operation and prove the equivalence of distributional and coarse-grained solutions. 
In section \ref{sec:CGNS} we derive balance equations for the coarse-grained compressible Navier-Stokes system.
In section \ref{sec:proofs} we establish auxiliary commutator estimates necessary for our main theorems.
In sections \ref{sec:proofThm1}--\ref{sec:proofThm3} we prove Theorems \ref{Onsager1E}--\ref{timeReg}.

%%%%%%%%%%%%%%%%%%%%%%%%%%%%%%%%%%%%%%%%%%%%%%%%%%%%%%%%%%%%%%%%%%%%%%%%%%%%%%%%%%%%%%%%%%%%%%%%%%%%%%%%%%%%%%%%%%%%%%%%%%%%%%%%%%%%%%%%%%%%%%%%%%%%%%%%%%%%%%%%%%%%%%%%%%%%%%%%%%%%%%%%%%%%%%%%%%%%%%%%%%%%%%%%%%

\section{Coarse-Grained Solutions and Weak Solutions}\label{sec:CGWS}

We are concerned in this section with general balance equations of the form 
\be \partial_t \bv + \nabla_x\cdot \bF=\bzed \label{gen-bal} \ee 
on a space-time domain $\Omega\times\mathbb{R}$ where again either 
$\Omega=\mathbb{T}^d$ or $\mathbb{R}^d,$ for simplicity, and $\bv\in \mathbb{R}^m$ and $\bF\in \mathbb{R}^{d\times m}.$
As usual, one defines $(\bv,\bF)$ to be a {\it weak/distributional solution} of (\ref{gen-bal}) iff
\be \langle \partial_t\varphi,\bv\rangle + \langle \nabla_x\varphi;\bF\rangle=\bzed, \quad \forall
\varphi\in D(\Omega\times \mathbb{R}), \label{gb-dist} \ee
where the space $D(\Omega\times \mathbb{R})=C_c^\infty(\Omega\times \mathbb{R})$ of test functions 
consists of $C^\infty$ functions $\varphi$ compactly supported in space-time, provided the  
topology defined by uniform convergence of functions and all their derivatives on compact 
sets containing all the supports.  Components $u_a,F_{ia}$ belong to the space $D'(\Omega\times \mathbb{R})$ 
of continuous linear functionals on $D(\Omega\times \mathbb{R})$, with  
$\langle \partial_t\varphi,\bv\rangle_a=\langle \partial_t\varphi,u_a\rangle$ and 
$\langle \nabla_x\varphi;\bF\rangle_a=\sum_{i=1}^d\langle \nabla_{x_i}\varphi,F_{ia}\rangle$ 
for $a=1,\dots,m.$ For these standard notions, e.g. see \cite{showalter2011hilbert,rudin2006functional}. 
We offer here a slightly different point of view on these topics. 

Let $\cG$ be a standard space-time \emph{mollifier}, with $\cG\in D(\Omega\times \mathbb{R}),$ $\cG\geq 0,$ and also
$\int_\Omega d^dr\int_{\mathbb{R}} \rmd \tau\ \cG(\br,\tau)=1.$ To simplify certain estimates we also assume, 
without loss of generality, that ${\rm supp}(\cG)$ is contained in the Euclidean unit ball in $(d+1)$ dimensions. 
Define the dilatation $\cG_\ell(\br,\tau)=\ell^{-(d+1)} \cG(\br/\ell, \tau/\ell)$ and space-time reflection 
$\check{\cG}(\br,\tau)=\cG(-\br,-\tau)$. For any $\bv\in D'(\Omega\times \mathbb{R})$ we define 
its \emph{coarse-graining at scale $\ell$} by 
\be \bar{\bv}_\ell=\check{\cG}_\ell *\bv\in C^\infty(\Omega\times \mathbb{R}). \label{cg-def} \ee 
Here $\ast$ denotes the convolution defined by 
\be (\check{\cG}_\ell *\bv)(\bx,t) = \langle S_{\bx,t}\cG_\ell,\bv\rangle \label{convdef1} \ee
for shift operator $(S_{\bx,t}\cG_\ell)(\br,\tau)=\cG_\ell(\br-\bx,\tau-t)$ or, equivalently, by 
\be \langle \varphi, \check{\cG}_\ell *\bv\rangle=\langle\varphi*\cG_\ell,\bv\rangle \label{convdef2} \ee  
for all test functions $\varphi\in D(\Omega\times \mathbb{R}).$ See \cite{rudin2006functional}. 
We say that $(\bv,\bF)$ are a {\it (space-time) coarse-grained solution} of (\ref{gen-bal}) iff 
\be \partial_t \bar{\bv}_\ell + \nabla_x\cdot \bar{\bF}_\ell=\bzed \label{gb-cg} \ee 
holds pointwise in space-time for all $\ell>0.$ We then have: 

\begin{proposition}\label{CGequiv}
$(\bv,\bF)$ are a distributional solution of (\ref{gen-bal}) on $\Omega\times \mathbb{R}$ iff 
$(\bv,\bF)$ are a coarse-grained solution of (\ref{gen-bal}) on $\Omega\times \mathbb{R}$
\end{proposition}

\begin{proof} 
If $(\bv,\bF)$ satisfy (\ref{gen-bal}) weakly, then taking $\varphi=S_{\bx,t}\cG_\ell$ in (\ref{gb-dist})
for any space-time point $(\bx,t)$ implies (\ref{gb-cg}) by the definition (\ref{convdef1}) of 
the convolution.

On the other hand, suppose that $(\bv,\bF)$ are a coarse-grained solution of (\ref{gen-bal}). 
Smearing (\ref{gb-cg}) with an arbitrary test function 
$\varphi\in D(\Omega\times \mathbb{R}),$ then gives by the second definition (\ref{convdef2}) of convolution that 
\be 
\langle (\partial_t\varphi)*\cG_\ell,\bv\rangle + \langle (\nabla_x\varphi)*\cG_\ell;\bF\rangle=0. \ee
However, in the limit $\ell\rightarrow 0,$ then  $(\partial_t\varphi)*\cG_\ell\rightarrow \partial_t\varphi$ 
and  $(\nabla_x\varphi)*\cG_\ell\rightarrow\nabla_x\varphi$ in the standard Fr\'echet topology on test functions.
Since $\bv,$ $\bF\in D'(\Omega\times\mathbb{R})$ are, by definition, continuous functionals on  $D(\Omega\times\mathbb{R}),$
the equation (\ref{gb-dist}) of the standard weak formulation immediately follows. 
\hfill $\Box$ \end{proof} 

This equivalence extends to solutions with prescribed initial-data. A standard approach to  
define weak solutions $(\bv,\bF)$ of (\ref{gen-bal}) on space-time domain $\Omega\times [0,\infty)$
with initial data $\bv_0\in D'(\Omega)$  is 
to require that 
\be \langle \partial_t\varphi,\bv\rangle + \langle \nabla_x\varphi;\bF\rangle 
+ \langle\varphi(\cdot,0),\bv_0\rangle=\bzed, \quad \forall \varphi\in D(\Omega\times [0,\infty)). 
\label{wksol-def-id} \ee
Here the space $D(\Omega\times [0,\infty))$ is taken to consist of piecewise-smooth 
functions of the form $\varphi(\bx,t)=\theta(t)\phi(\bx,t),$ products of the Heaviside step function $\theta(t)$ and 
some $\phi\in D(\Omega\times \mathbb{R}).$ Such test functions $\varphi\in D(\Omega\times [0,+\infty))$ 
are \emph{causal}, with $\varphi(\bx,t)=0$ for $t<0.$ In order to make the lefthand side of (\ref{wksol-def-id})
meaningful, a stronger assumption is required than only $(\bv,\bF)\in D'(\Omega\times\mathbb{R})$. 
A very general assumption is that distributional products $\theta\odot \bv,$ $\theta\odot\bF$ exist defined by 
$\theta\odot f :=\Dlim_{\ell\to 0} \theta \ol{f}_\ell$ for $f\in D'(\Omega\times\mathbb{R})$
\cite{Oberguggenberger92}. In that case, we can take 
\be\label{definLowReg}
 \langle \partial_t\varphi,\bv\rangle :=  \langle \partial_t \phi, \theta\odot \bv\rangle, 
 \quad  \langle \nabla_x\varphi;\bF\rangle := \langle \nabla_x\phi;\theta\odot \bF\rangle. 
\ee
Because limit distributions $\theta\odot f$ clearly have support in $\Omega\times [0,\infty),$ the 
definition (\ref{definLowReg}) does not depend upon the choice of $\phi$ such that $\varphi=\theta\phi.$
In the special case when $f=\bv,\bF\in L_{loc}^1(\Omega\times [0,\infty))$, then strong convergence 
of $\ol{f}_\ell\to f$ in $ L_{loc}^1$ (e.g. see Lemma 7.2 of \cite{gilbarg2015elliptic}) implies that the definitions 
(\ref{definLowReg}) reduce to their standard interpretation.  
In addition,to make the definition (\ref{wksol-def-id}) meaningful, 
one must require weak-$\ast$ continuity of the distribution $\bv$ in time, so that $t\mapsto\langle \psi,\bv(\cdot,t)\rangle$ 
is continuous for all $\psi\in D(\Omega).$ Initial data is then achieved in the sense that 
\be \lim_{t\rightarrow 0+} \langle \psi,\bv(\cdot,t)\rangle=\langle \psi,\bv_0\rangle, \quad \forall \psi\in D(\Omega).
\label{dist-id} \ee
The coarse-graining approach can be also carried over with only minor changes. 
The mollifier $\cG$ must now be chosen to be \emph{strictly causal}, with $\cG\in D(\Omega\times (0,\infty))$
and thus $\cG(\br,\tau)\equiv 0$ for $\tau\leq 0.$ The definition (\ref{cg-def}) of coarse-graining still applies, noting that the convolution in 
time is $(\chi_1*\chi_2)(t)=\int_0^t \rmd s\ \chi_1(s)\chi_2(t-s)$ for causal functions $\chi_1,\chi_2.$
We can again define $(\bv,\bF)$ to be a coarse-grained solution of (\ref{gen-bal}) if (\ref{gb-cg}) holds 
pointwise in space-time for all $\ell>0.$ Since $\ol{\bv}_\ell\in C^\infty(\Omega\times [0,\infty))$,
the functions $\ol{\bv}_\ell(\cdot,0)\in C^\infty(\Omega)$ are well-defined and 
the coarse-grained solution is naturally said to take on initial data $\bv_0\in D'(\Omega)$ when 
\be \Dlim_{\ell\to 0}\ol{\bv}_\ell(\cdot,0)=\bv_0. \label{cg-id} \ee  
It is straightforward to see for all $\psi\in D(\Omega)$ that 
\be \langle \psi, \ol{\bv}_\ell\rangle = \int d^dr \int_0^\infty d\tau\ \cG_\ell(\br,\tau) \Psi(\br,t), \quad
\Psi(\br,\tau):=\langle S_{\br}\psi,\bv(\cdot,\tau)\rangle.  \label{Psi-def} \ee
Suppose that one requires not only weak-$\ast$ continuity of $\bv$ in time, but also the stronger statement 
that $\Psi(\br,\tau)$ defined in (\ref{Psi-def}) is jointly continuous in $(\br,\tau)$ for all $\psi\in D(\Omega).$
The initial data prescribed by (\ref{definLowReg}) and (\ref{cg-id}) are then the same. 

This leads to: 
\begin{proposition}
If $(\bv,\bF)$ is a coarse-grained solution of (\ref{gen-bal}) on $\Omega\times [0,\infty)$ with initial data $\bv_0,$
then it is a distributional solution with the same initial data. If also $\langle S_{\br}\psi,\bv(\cdot,\tau)\rangle$ is 
jointly continuous in $(\br,\tau)$ for all $\psi\in D(\Omega),$ then a distributional solution $(\bv,\bF)$ of (\ref{gen-bal}) 
on $\Omega\times [0,\infty)$ with initial data $\bv_0$ is a coarse-grained solution with the same initial data. 
\end{proposition}
\begin{proof}
To prove the first statement, multiply the coarse-grained equation (\ref{gb-cg}) with the Heaviside 
function $\theta$ and then smear with an arbitrary $\phi\in 
D(\Omega\times \mathbb{R}).$ An integration-by-parts in time gives that 
$$ \langle (\partial_t\phi),\theta\overline{\bv}_\ell\rangle + \langle (\nabla_x\phi);\theta\overline{\bF}_\ell\rangle
+\langle \phi(\cdot,0), \ol{\bv}_\ell\rangle =0. $$
Taking the limit $\ell\rightarrow 0$ with definition (\ref{definLowReg}) and assumption (\ref{cg-id}) 
recovers (\ref{wksol-def-id}).

For the second statement,  take $\varphi=S_{\bx,t}\cG_\ell \in D(\Omega\times (0,\infty))$
for any $\bx\in\Omega$ and $t\geq 0.$ 
We see that $\varphi$ is strictly causal, i.e. $\varphi(\cdot,0)=0.$ The equation 
(\ref{wksol-def-id}) of the weak formulation thus yields the coarse-grained equation (\ref{gb-cg})  
for that choice of $(\bx,t)$ and $\ell.$  Furthermore, because of (\ref{Psi-def}) 
and the joint continuity of $\langle S_{\br}\psi,\bv(\cdot,\tau)\rangle$ in $(\br,\tau),$
$\ol{\bv}_\ell(\cdot,0)\Dto \bv_0$ holds for the same $\bv_0$ given by (\ref{dist-id}).  \hfill $\Box$ \end{proof}

\begin{remark}
If $\bv \in C([0,\infty);L^p(\Omega))$ with continuity in the strong $L^p$-norm topology
for some $p\geq 1$, then the joint continuity follows from the obvious continuity of $\Psi(\br,\tau)$ 
in $\br$ for each $\tau$ and the H\"older inequality 
$$|\Psi(\br,\tau)-\Psi(\br,\tau')|\leq \|\psi\|_q\|\bv(\cdot,\tau)-\bv(\cdot,\tau')\|_p, \quad q=p/(p-1), $$
which implies continuity of $\Psi(\br,\tau)$ in $\tau$ uniform in $\br\in \Omega.$
\end{remark}

\begin{remark} In Lemma 8 of \cite{DeLellisSzekelyhidi10} it was proved that, if 
$(\bv,\bF)$ is a weak solution  with $\bv\in L^\infty([0,\infty),L^2(\Omega))$ and 
$\bF\in L^1_{\rm loc}(\Omega\times [0,\infty)),$ then $\bv$ can always be altered 
on a zero measure set of times so that $\bv\in C_w([0,\infty),L^2(\Omega)),$ 
with continuity in the weak topology of $L^2(\Omega).$ In that case, 
$\Psi(\br,\tau)$ defined for any $\psi\in D(\Omega)$ by (\ref{Psi-def}) is continuous 
in $\tau$ for each $\br\in\Omega.$  By Cauchy-Schwartz, 
$$ |\nabla_r \Psi(\br,\tau)| \leq \|\nabla\psi\|_2  \|\bv\|_{L^\infty([0,\infty);L^2(\Omega))}, $$
%{\rm ess.sup}_{\tau\in [0,\infty)}\|\bv(\tau)\|_{L^2(\Omega)},$$
so that $\Psi(\br,\tau)$ is also (Lipschitz) continuous in $\br$ uniformly in $\tau,$
and thus is jointly continuous in $(\br,\tau)$ under the same assumptions as in \cite{DeLellisSzekelyhidi10}. 
\end{remark} 

\begin{remark}\label{rem:finiteTime}
The above results hold with only minor modifications for solutions on $\Omega\times [0,T)$ 
with $0<T<\infty.$  Coarse-grained solutions are required now to satisfy 
equations (\ref{gb-cg}) only for $\bx,t$ and $\ell$ such that $S_{\bx,t}\cG_\ell\in D(\Omega\times (0,T)).$ 
On the other hand, for any $\varphi\in D(\Omega\times [0,T)),$ then $T_\varphi=\max\{t: (\bx,t)\in {\rm supp}(\varphi)\}<T$.
Since ${\rm supp}(\cG)$ is contained in the unit ball, then $S_{\bx,t}\cG_\ell\in D(\Omega\times (0,T))$ 
for any $\ell<T-T_\varphi$ and  $(\bx,t)\in {\rm supp}(\varphi)$ and our previous arguments on equivalence
of the two notions of solution can be repeated without change. 
\end{remark}

\begin{remark} In the paper \cite{CET1994}, only space mollification was employed. 
One can also define a space coarse-graining with a standard mollifier $G_\ell(\br)=\ell^{-d}G(\br/\ell),$
that is, $ \hat{\bv}_\ell=\check{G}_\ell *\bv.$ This is a smooth function of space but only a distribution in time. 
In that case, we say that $(\bv,\bF)$ are a {\it (space) coarse-grained solution} of the balance relation 
(\ref{gen-bal}) iff 
\be \partial_t \hat{\bv}_\ell + \nabla_x\cdot \hat{\bF}_\ell=\bzed \label{gb-scg} \ee
holds pointwise in space and distributionally in time for all $\ell>0$. This is also equivalent to the standard notion of 
weak solution, as can be seen by arguments very similar to those given above. 
If furthermore $\bv,$ $\bF\in L^1_{loc}(\Omega\times (0,T)),$ then standard approximation arguments show that 
the time-derivative in (\ref{gb-scg}) can be taken to be a classical derivative at Lebesgue almost all times. 
\end{remark} 

In many applications, including those considered in this paper, $\bv$ is not merely a distribution but a 
measurable function of space-time, and $\bF:=\bF(\bv)$ is a pointwise nonlinear function of $\bv.$ A key aspect of 
the coarse-graining operation is that coarse-graining nonlinear functions of fields generally gives a result different 
from evaluating the 
function at the coarse-grained fields, i.e. the operations of coarse-graining and function-evaluation do not commute.
For simple products of the form $f_1 f_2 \cdots f_n,$ this non-commutation can be measured by 
\emph{coarse-graining cumulants}, which are defined iteratively in $n$ by $\tau_\ell(f)=\bar{f}_\ell$ and
\be  \overline{(f_1 \cdots f_n)}_\ell = \sum_\Pi  \prod_{p=1}^{|\Pi|} \ol{\tau}_\ell(f_{i^{(p)}_1},\dots,f_{i^{(p)}_{n_p}}),  \ee
where the sum is over all partitions $\Pi$ of the set $\{1,2,\dots,n\}$ into $|\Pi|$ disjoint subsets $\{ i^{(p)}_1,\dots,i^{(p)}_{n_p}\},$
$p=1,\dots,|\Pi|.$   See e.g. \cite{huang2009introduction,stuart2009kendall}.  For example, for $n=2$
\be
\ol{(f g)}_\ell  = \ol{f}_\ell\ol{g}_\ell+\ol{ \tau}_\ell ({f}, g) \ \ \ \  {\rm or} \ \ \ \ol{ \tau}_\ell ({f}, g) = \ol{(f  g)}_\ell-\ol{f}_\ell \ol{g}_\ell.
\ee
For general composed functions $h=h(f_1, \cdots, f_n)$ with $h$ a smooth nonlinear function on $\mathbb{R}^n$, the non-commutation 
is measured by the quantity 
\be  \Delta_\ell h:= \overline{h(f_1, \cdots, f_n)}_\ell - h(\ol{(f_1)}_\ell, \cdots, \ol{(f_n)}_\ell). \ee 
To simplify the writing of various expressions, we shall  often use an ``under-bar" notation to indicate the function 
evaluated at coarse-grained fields: 
\be\label{underbarNot}
\ul{h}_\ell:=h(\ol{(f_1)}_\ell, \cdots, \ol{(f_n)}_\ell),
\ee
whereas $\ol{h}_\ell= \overline{h(f_1, \cdots, f_n)}_\ell.$ 

\begin{remark}
If, as in Remark \ref{rem:finiteTime} above, we consider space-time domains with a
finite time interval $\Gamma= \Omega\times (0,T)$, $T<\infty$ (or a semi-infinite interval $\Omega\times (0,\infty)$ for
mollifiers which are not causal), coarse-graining cumulants $\tau_\ell(f_1, \cdots, f_n)$ and smooth functions $\ul{h}_\ell$ 
of coarse-grained fields are not defined everywhere on $\Gamma$ for $\ell>0$. Instead, they are defined only for $(\bx,t)\in \Gamma$ 
such that $S_{\bx,t} \cG_\ell \in \mathcal{D}( \Omega\times (0,T)),$ e.g. when the distance of $(\bx,t)$ to $\partial\Gamma$ 
is less than $\ell.$ They are thus well-defined for every $(\bx,t)\in \Omega\times (0,T)$ at sufficiently small $\ell.$
\end{remark}

%%%%%%%%%%%%%%%%%%%%%%%%%%%%%%%%%%%%%%%%%%%%%%%%%%%%%%%%%%%%%%%%%%%%%%%%%%%%%%%%%%%%%%%%%%%%%%%%%%%%%%%%%%%%%%%%%%%%%%%%%%%%%%%%%%%%%%%%%%%%%%%%%%%%%%%%%%%%%%%%%%%%%%%%%%%%%%%%%%%%%%%%%%%%%%%%%%%%%%%%%%%%%%%%%%%%%%%%%%%%%%%%%%%%%%%%%%%%%%%%%%%%%%%%%%%%%%%%%%%%%%%%%%%%%%%%%%%%%%%%%%%%%%%%%%%%%%

\section{Coarse-Grained Navier-Stokes and Balance Equations} \label{sec:CGNS}

We now discuss the results of coarse-graining the solutions of the compressible Navier-Stokes system. None of the results 
in this section depend upon the particular type of coarse-graining and are valid whether coarse-graining is in 
space, time, space-time or using some other averaging procedure (such as as weighted coarse-graining). 
We drop the superscript $\varepsilon$ in this section to simplify notations. 

The coarse-grained Navier-Stokes equations for mass density $\vr,$ momentum density $\bj=\vr\bu,$ 
and energy density $E$ are  
\begin{eqnarray} \label{eq:CGrho}
\partial_t \ol{\vr}_\ell &+& \nabla_x\cdot\ol{\nj}_\ell =0,\\
 \label{eq:CGj}
\partial_t \ol{\,\nj\,}_\ell &+& \nabla_x \cdot \left( \ol{(\bj  \bu)}_\ell +  \ol{p}_\ell \bI +\ol{\bT}_\ell \right)  = \bzed, \\
\label{eq:CGTE}
\partial_t \ol{E}_\ell &+& \nabla_x \cdot \left( \ol{((E+p)\bu)}_\ell +\ol{(\bT\cdot \bu)}_\ell + \ol{\bq}_\ell \right)  =0.
\end{eqnarray}
It is useful to rewrite the equations (\ref{eq:CGrho}) and (\ref{eq:CGj}) employing the \emph{Favre (density-weighted) averaging}: 
\be \tilde{f}_\ell = \ol{(\vr f)}_\ell/\ol{\vr}_\ell. \ee 
One may likewise define cumulants $\tilde{\tau}_\ell(f_i,\dots,f_n)$ with respect to this Favre filtering. 
See \cite{favre1969statistical,aluie2013scale}. With this new averaging, (\ref{eq:CGrho})--(\ref{eq:CGj}) may be rewritten: 
\begin{eqnarray} \label{eq:CGrho2}
\partial_t \ol{\vr}_\ell &+& \nabla_x\cdot\left( \ol{\vr}_\ell \tilde{\bu}_\ell \right) =0,\\
\label{eq:CGu}
\ol{\vr}_\ell( \partial_t  + \tilde{\bu}_\ell\cdot\nabla_x)\tilde{\bu}_\ell &+& \nabla_x\cdot\left( \ol{\vr}_\ell\tilde{\tau}_\ell (\bu, \bu) +  \ol{p}_\ell \bI +\ol{\bT}_\ell \right)  = 0.
\end{eqnarray}
We emphasize that our use of Favre coarse-graining is mathematically only a matter of convenience, in order to reduce 
the number of terms in our coarse-grained equations (and to provide them with simple physical interpretations \cite{aluie2013scale,eyinkdrivascompr}). 
Favre cumulants of $f_1,\dots,f_n$ may always be rewritten in terms of unweighted cumulants of $f_1,\dots,f_n$ and $\vr.$
For example \cite{aluie2013scale,eyinkRec2015}:
\begin{eqnarray} \label{favref}
&& \qquad \hspace{12pt} \tilde{f}_\ell  = \ol{f}_\ell + \frac{1}{\ol{\vr}_\ell} \ol{\tau}_\ell(\vr, f),\\
\label{favreTau}
&& \quad \tilde{\tau}_\ell(f,g)  = \ol{\tau}_\ell(f,g) + \frac{1}{\ol{\vr}_\ell}\ol{\tau}_\ell(\vr,f,g) - \frac{1}{\ol{\vr}_\ell^2}\ol{\tau}_\ell(\vr,f)\ol{\tau}_\ell(\vr,g) ,\\
\label{favreTau3}
&& \tilde{\tau}_\ell(f,g,h)  = \ol{\tau}_\ell(f,g,h) + \frac{1}{\ol{\vr}_\ell}\ol{\tau}_\ell(\vr,f,g,h) \\
\nonumber
&& \quad \ \ - \frac{1}{\ol{\vr}_\ell^2} 
[\ol{\tau}_\ell(\vr,f)\ol{\tau}_\ell(\vr,g,h) + \mbox{cyc. perm. $f,g,h$}] 
+ \frac{2}{\ol{\vr}_\ell^3}\ol{\tau}_\ell(\vr,f)\ol{\tau}_\ell(\vr,g)\ol{\tau}_\ell(\vr,h).
\end{eqnarray}
We next derive various balance equations for the coarse-grained fields. 
\vspace{2mm}

 %%%%%%%%%%%%%%%%%%%%%%%%%%%%%%

\noindent \emph{Resolved Kinetic Energy:} Following Aluie \cite{aluie2013scale}, we consider a resolved kinetic energy 
$ \frac{1}{2}\ol{\vr}_\ell |\tilde{\bu}|^2= |\ol{\nj}|_\ell^2/2\ol{\vr}_\ell$.   
%\red{Note that, because $|\bu|^2$ is a convex function of $\bu$, Jenson's inequality implies that:\begin{equation}\label{JensonKE}
%\frac{1}{2} \ol{\vr}_\ell |\tilde{\bu}_\ell |^2\leq \frac{1}{2} \ol{\vr}_\ell \widetilde{|\bu |^2}_\ell = \frac{1}{2}\ol{\left(\vr |\bu|^2\right)}_\ell.
%\end{equation}
%Thus, the integral over space of $\frac{1}{2} \ol{\vr}_\ell |\tilde{\bu}_\ell |^2$ is less than the total kinetic energy and represents only the resolved component.}
Using (\ref{eq:CGrho2}) and (\ref{eq:CGu}) one can derive its balance equation:
\be
\partial_t \left( \frac{1}{2} \ol{\vr}_\ell |\tilde{\bu}_\ell |^2\right) + \nabla_x \cdot \bJ_\ell^v =    \ol{p}_\ell \ol{\Theta}_\ell -Q_\ell^{\rm flux}   -D_\ell^v ,\label{eq:CGKE}
\ee
where the various terms are defined by:
 \begin{eqnarray}  \label{eq:Juell}
\bJ_\ell^v &:=& \left(\frac{1}{2}\ol{\vr}_\ell|\tilde{\bu}_\ell|^2 
+\ol{p}_\ell\right)\tilde{\bu}_\ell+ \ol{\vr}_\ell\tilde{\bu}_\ell\cdot\tilde{ \tau}_\ell (\bu, \bu)  -\frac{\ol{p}_\ell }{\ol{\vr}_\ell }\ol{ \tau}_\ell(\vr, \bu)  + \tilde{\bu}_\ell\cdot \ol{\bT}_\ell ,\\
 \label{eq:Quell}
Q_\ell^{\rm flux} &:=& \frac{\nabla_x\ol{p}_\ell }{\ol{\vr}_\ell } \cdot \ol{ \tau}_\ell(\vr, \bu)- \ol{\vr}_\ell \nabla_x \tilde{\bu}_\ell :\tilde{ \tau}_\ell(\bu,\bu)   ,\\
D_\ell^v &:=& -\nabla_x\tilde{\bu}_\ell :\ol{\bT}_\ell    .\label{eq:Vuell}
 \end{eqnarray}
Equation (\ref{eq:CGKE}) may be rewritten as 
\be
\partial_t \left( \frac{1}{2} \ol{\vr}_\ell |\tilde{\bu}_\ell |^2\right) + \nabla_x \cdot \bJ_\ell^v =    \ol{(p\Theta)}_\ell -Q_\ell^{\rm inert}   -D_\ell^v, \label{eq:CGKE2}
\ee
where the ``inertial dissipation'' is defined by 
\be Q_\ell^{\rm inert} :=Q_\ell^{\rm flux} +\ol{\tau}_\ell(p,\Theta). \ee

 %%%%%%%%%%%%%%%%%%%%%%%%%%%%%%

\noindent \emph{Unresolved Kinetic Energy}. 
We define this quantity (with summation over repeated $i$ indices) as 
\be\label{subscaleKE}
k_\ell:=\frac{1}{2} \ol{\vr}_\ell \tilde{\tau}_\ell(v_i,v_i). 
\ee
Note that 
%\begin{equation}\label{subscaleKEV2}
$\frac{1}{2} \ol{\vr}_\ell |\tilde{\bu}_\ell |^2+ k_\ell=\frac{1}{2}\ol{\left(\vr |\bu|^2\right)}_\ell,$
%\end{equation}
whose integral over $\Omega$ is a time-mollification of the total kinetic energy. 
Taking the difference of the coarse-grained kinetic-energy Eq. (\ref{eq:KE}) governing $\frac{1}{2}\ol{\left(\vr |\bu|^2\right)}_\ell$ and Eq. (\ref{eq:CGKE}) for $\frac{1}{2} \ol{\vr}_\ell |\tilde{\bu}_\ell |^2$,  one obtains:
 \be
\partial_t k_\ell+ \nabla \cdot \bJ_\ell^k=  (\ol{\tau}_\ell(p,\Theta)-\ol{Q}_\ell)+Q_\ell^{\rm flux} 
+D_\ell^k, \label{eq:taupD}  
 \ee
 where 
 \begin{eqnarray}
 \label{Jkell} 
 \bJ_\ell^k :&=& \frac{1}{2} \ol{\vr}_\ell \tilde{\tau}_\ell(v_i,v_i)\tilde{\bu}_\ell + \ol{\tau}_\ell (p,\bu)+ \frac{1}{2} \ol{\vr}_\ell \tilde{\tau}_\ell(v_i,v_i,\bu) \\
 && \hspace{80pt} + \ol{(\bT\cdot\bu)}_\ell-\ol{\bT}_\ell\cdot \tilde{\bu}_\ell , \nonumber \\
 \label{Dkell}
 D_\ell^k :&=& -\ol{\bT}_\ell:\nabla_x\tilde{\bu}_\ell.
 \end{eqnarray}
 
  %%%%%%%%%%%%%%%%%%%%%%%%%%%%%%

\noindent \emph{Resolved Internal Energy:} Directly coarse-graining equation (\ref{eq:IE}), one finds the following 
balance equation for the resolved internal energy: 
\be 
\partial_t \ol{u}_\ell + \nabla_x \cdot \bJ_\ell^u = \ol{Q}_\ell- \ol{(p\Theta)}_\ell ,
\label{eq:CGIE}
\ee 
where 
 \be
 \label{eq:Jeell}
\bJ_\ell^u = \ol{(u\bu)}_\ell+\ol{\bq}_\ell  = \ol{u}_\ell\ol{\bu}_\ell+\ol{ \tau}_\ell (u,\bu)+\ol{\bq}_\ell .
\ee
%\red{Because  $\ol{u}_\ell$ is coarse-grained, it represents a large-scale or resolved feature.  However, the equation (\ref{eq:CGIE}) governing its evolution contains the (coarse-grained) viscous dissipation $\ol{Q}_\ell$ which will be non-vanishing even at fixed $\ell$ in the ideal limit if there is a dissipative anomaly. Therefore the evolution of $\ol{u}_\ell$ will directly depend on the details of the dissipative dynamics.}  
A more important quantity for our analysis is $\ol{u}_\ell^*:= \ol{u}_\ell+ k_\ell,$ which we term the ``intrinsic resolved internal energy''.
It is defined more fundamentally by the  implicit relation 
\be \ol{E}_\ell = \frac{1}{2} \ol{\vr}_\ell|\tilde{\bu}_\ell |^2+ \ol{u}_\ell^*, \label{LSu*} \ee 
in terms of the resolved quantities $\ol{\vr}_\ell,$ $\tilde{\bu}_\ell$, and $\ol{E}_\ell$.
%This is a natural object because, based on coarse-grained observations alone, it is impossible to distinguish between energy in thermal fluctuations (resolve internal energy $\ol{u}_\ell$) and in unresolved turbulent fluctuations (unresolved kinetic energy $k_\ell$). 
%Note that  by (\ref{subscaleKEV2}), $k_\ell$ is non-negative and therefore $\ol{u}_\ell^*$ shares this property.
One thus derives a balance equation for this intrinsic internal energy by subtracting the resolved kinetic energy 
balance (\ref{eq:CGKE}) from the coarse-grained total energy equation (\ref{eq:CGTE}): 
\be 
\partial_t \ol{u}_\ell^* + \nabla_x \cdot \bJ_\ell^{u*} =   Q_\ell^{\rm flux} -\ol{p}_\ell \ol{\Theta}_\ell+  D_\ell^k,
\label{eq:CGIE*}
\ee 
where $ D_\ell^k$ is defined by equation (\ref{Dkell}) and 
 \begin{eqnarray}  \nonumber
\bJ_\ell^{u*} &=& \ol{u}_\ell\ol{\bu}_\ell+\ol{ \tau}_\ell (h,\bu)  + \frac{1}{2} \ol{\vr}_\ell \tilde{\tau}_\ell(v_i,v_i)\tilde{\bu}_\ell + \frac{1}{2} \ol{\vr}_\ell \tilde{\tau}_\ell(v_i,v_i,\bu) \\
& &\    +\ \ol{\bq}_\ell+ \ol{(\bT\cdot\bu)}_\ell-\ol{\bT}_\ell\cdot \tilde{\bu}_\ell,
 \end{eqnarray} \label{eq:Jeell*}
with $h:=u+p$ defining the standard thermodynamic enthalpy.  %\red{The terms equation (\ref{eq:CGIE*}) governing do not directly depend on the viscous dynamics in the ideal limit.  Therefore, the evolution of $\ol{u}_\ell^*$ can be defined for general Euler solutions.   This point is discussed again below in Remark \ref{RemarkIntrinsic}. }
 \\

 %%%%%%%%%%%%%%%%%%%%%%%%%%%%%%
\vspace{-2mm}
\noindent \emph{Resolved Entropy}: %\red{We define the resolved entropy as the quantity $\ul{s}_\ell:=s(\ol{u}_\ell,\ol{\vr}_\ell)$.  %Recalling that $s(u,\vr)$ is a concave function of the $u,\vr$ (see references and discussion in \cite{eyinkdrivascompr}), it follows from Jensen's inequality that:
%\begin{equation}
%\ol{s(u,\vr)}_\ell \leq s(\ol{u}_\ell,\ol{\vr}_\ell).
%\end{equation}
%Thus, spatial coarse-graining increases entropy since the integral over space of $\ul{s}_\ell$ is greater than the total entropy, and we regard $\ul{s}_\ell$ as the resolved component.}  
We derive an equation for $\ul{s}_\ell:=s(\ol{u}_\ell,\ol{\vr}_\ell)$ using (\ref{eq:CGIE}), also (\ref{eq:CGrho}) rewritten as
\be \partial_t \ol{\vr}_\ell +\nabla_x\cdot(\ol{\vr}_\ell\ol{\bu}_\ell+\ol{\tau}_\ell(\vr,\bu)) =0,\ee 
the homogeneous Gibbs relation $\ul{T}_\ell \ul{s}_\ell = (\ol{u}_\ell + \ul{p}_\ell)- \ul{\mu}_\ell \ol{\vr}_\ell$, and the first law of thermodynamics: 
\be\label{eq:CGGibbs}
 \ul{T}_\ell\ol{\mathcal{D}}_t \ul{s}_\ell= \ol{\mathcal{D}}_t\ol{u}_\ell - \ul{\mu}_\ell \ol{\mathcal{D}}_t \ol{\vr}_\ell,
\ee
with $\ol{\mathcal{D}}_t =\partial_t + \ol{\bu}_\ell\cdot \nabla$ being the material derivative along the smoothed flow. One then finds that the resolved entropy satisfies:
\be
\partial_t \ul{s}_\ell+ \nabla_x \cdot \bJ_\ell^s =   \frac{\ol{Q}_\ell-\ol{\tau}_\ell(p,\Theta)}{\ul{T}_\ell}  -  I_\ell^{\rm flux}+  \Sigma_\ell^{\rm flux}+ D_\ell^s, \label{eq:LSs}
\ee
where
 \begin{eqnarray}  \label{eq:Jsell}
\bJ_\ell^s &:=& \ul{s}_\ell \ol{\bu}_\ell+ \ul{\beta}_\ell\left( {\ol{\tau}_\ell (u,\bu)} + {\ol{\bq}_\ell}\right) -\ul{\lambda}_\ell\ol{ \tau}_\ell(\vr, \bu),\\
 \label{eq:Isell}
I_\ell^{\rm flux}&:=& \ul{\beta}_\ell (\ol{p}_\ell -\ul{p}_\ell) \ol{\Theta}_\ell,\\
 \label{eq:Lambell}
 \Sigma_\ell^{\rm flux}&:=&  \nabla_x \ul{\beta}_\ell \cdot\ol{\tau}_\ell (u,\bu) - \nabla_x \ul{\lambda}_\ell \cdot\ol{ \tau}_\ell(\vr,\bu),\\
D_\ell^s &:=& -\frac{ \ol{\bq}_\ell\cdot   \nabla_x\ul{T}_\ell}{\ul{T}_\ell^2},\label{eq:Vsell}
 \end{eqnarray}
with $\beta := 1/T$ and $\lambda := \mu/T$. Considering the source terms on the righthand side of (\ref{eq:LSs}), we shall see that all of the 
terms marked ``flux'' satisfy simple bounds, and the direct dissipation term $D_\ell^s$ will be seen to vanish as 
$\varepsilon\rightarrow 0,$ but the quantity $\ol{Q}_\ell-\ol{\tau}_\ell(p,\Theta)$, 
which originates from the $\ol{\mathcal{D}}_t\ol{u}_\ell$ term in (\ref{eq:CGGibbs}), is more difficult to estimate. 
Fortunately, the same term appears in the balance equation for ``unresolved kinetic energy.''
\vspace{2mm}

  %%%%%%%%%%%%%%%%%%%%%%%%%%%%%%

\noindent \emph{Intrinsic Resolved Entropy}: In order to cancel the difficult term $\ol{Q}_\ell-\ol{\tau}_\ell(p,\Theta)$,  we introduce 
an ``intrinsic resolved entropy density'' by $\ul{s}_\ell^*:=s(\ol{u}_\ell,\ol{\vr}_\ell)+ \ul{\beta}_\ell k_\ell.$  
This quantity is defined more fundamentally by 
\be \ul{s}_\ell^* = \ul{\beta}_\ell(\ol{u}_\ell^*+\ul{p}_\ell)-\ul{\lambda}_\ell\ol{\vr}_\ell, \label{LSs*} \ee 
where $\ol{u}_\ell^*$ is the intrinsic resolved internal energy 
defined in (\ref{LSu*}). The two definitions are seen to be the same using 
the homogenous Gibbs relation (\ref{eq:homogenousGibbs}), or  $\ul{s}_\ell = \ul{\beta}_\ell(\ol{u}_\ell+\ul{p}_\ell)-\ul{\lambda}_\ell\ol{\vr}_\ell$.
By means of (\ref{LSs*}) and (\ref{eq:CGIE*}), together with the standard thermodynamic relation 
$\ol{\mathcal{D}}_t (\ul{\beta}_\ell\ul{p}_\ell)= \ol{\vr}_\ell \ol{\mathcal{D}}_t\ul{\lambda}_\ell - \ol{u}_\ell \ol{\mathcal{D}}_t \ul{\beta}_\ell,$ 
one obtains
\be\label{eq:thermoRel}
\ol{\mathcal{D}}_t \ul{s}_\ell^*= (\ol{\mathcal{D}}_t\ul{\beta}_\ell) k_\ell + \ul{\beta}_\ell \ol{\mathcal{D}}_t\ol{u}_\ell^* - \ul{\lambda}_\ell \ol{\mathcal{D}}_t \ol{\vr}_\ell.
\ee
rather than (\ref{eq:CGGibbs}). Note that $\ol{\mathcal{D}}_t\ol{u}_\ell^*$ appears here rather than $\ol{\mathcal{D}}_t\ol{u}_\ell$.  It is 
straightforward using (\ref{eq:thermoRel}) to derive the balance equation for $\ul{s}_\ell^*$: 
\be
\partial_t \ul{s}_\ell^*+ \nabla_x \cdot \bJ_\ell^{s*} =   
 - I_\ell^{\rm flux} +  \Sigma_\ell^{\rm flux *}+ D_\ell^s +\ul{\beta}_\ell D_\ell^k \label{eq:LSs*}
\ee
with 
 \begin{eqnarray}  \label{eq:Jsell*}
\bJ_\ell^{s*} &:=& \bJ_\ell^s + \ul{\beta}_\ell \bJ_\ell^k,\\
  \label{eq:Lambell*}
 \Sigma_\ell^{\rm flux *}&:=&  \Sigma_\ell^{\rm flux} +  \ul{\beta}_\ell Q_\ell^{\rm flux} + \partial_t\ul{\beta}_\ell\ k_\ell + \nabla_x\ul{\beta}_\ell \cdot \bJ_\ell^k  .
 \end{eqnarray}
We also then write 
\be \Sigma_\ell^{\rm inert *} = - I_\ell^{\rm flux} +  \Sigma_\ell^{\rm flux *} \label{eq:SigInert}\ee
for the net ``inertial'' production of the intrinsic entropy. 
The balance equation (\ref{eq:LSs*})  of the intrinsic entropy turns out to be the key identity for the proof of Theorem \ref{Onsager2}. 
On the righthand side, the direct dissipation terms will be shown to vanish as $\varepsilon\rightarrow 0$ and the remaining 
terms are ``flux-like'' and depend only upon increments of the basic variables $u,$ $\vr,$ $\bu.$ This latter result follows 
from commutator estimates of Section \ref{sec:proofs}. 

\begin{remark}\label{RemarkIntrinsic}
{Note that the balance equations (\ref{eq:CGKE}) for resolved kinetic energy, (\ref{eq:CGIE*}) for intrinsic resolved internal energy 
and (\ref{eq:LSs*}) for intrinsic resolved entropy are valid for general weak Euler solutions after setting 
${\bf T}={\bf q}={\bf 0},$
without the need for considering the viscous regularization with $\ve>0$ and taking $\ve\to 0.$ On the other hand, the balance equations 
(\ref{eq:taupD}) for unresolved kinetic energy, (\ref{eq:CGIE}) for resolved internal energy, and (\ref{eq:LSs}) for resolved 
entropy are valid with ${\bf T}={\bf q}={\bf 0}$
only for weak Euler solutions obtained from the inviscid limit. In fact, the latter equations 
contain the quantities $\ol{Q}_\ell$ and $\ol{\tau}_\ell(p,\Theta)$ which are {\it a priori} undefined for general weak Euler solutions.} 
\end{remark}  

%%%%%%%%%%%%%%%%%%%%%%%%%%%%%%%%%%%%%%%%%%%%%%%%%%%%%%%%%%%%%%%%%%%%%%%%%%%%%%%%%%%%%%%%%%%%%%%%%%%%%%%%%%%%%%%%%%%%%%%%%%%%%%%%%%%%%%%%%%%%%%%%%%%%%%%%%%%%%%%%%%%%%%%%%%%%%%%%%%%%%%%%%%%%%%%%%%%%%%%%%%%%%%%%%%%%%%%%%%%%%%%%%%%%%%%%%%%%%%%%%%%%%%%%%%%%%%%%%%%%%%%%%%%%%%%%%%%%%%%%%%%%%%%%%%%%%%

\section{Commutator Estimates}\label{sec:proofs}

The estimates that we derive in this section are valid for coarse-graining in space, time, or space-time. We state 
them here for the space-time coarse-graining that we use in our proofs of 
Theorems \ref{Onsager1E}--\ref{Onsager2}.
The need for coarse-graining in time as well as in space is due to the time-derivative term in expression (\ref{eq:Lambell*}) 
for $\Sigma_\ell^{\rm flux *}.$ In order to present the estimates, it is useful to employ a ``space-time vector'' notation, with 
$X=(\bx,c t),$ $R=(\br,c\tau)$ where $c$ is a constant with dimensions of velocity which is fixed independent of 
$\epsilon$ and $\ell.$ For example, we may take $c$ to be the speed of sound (or, in the relativistic case, the speed 
of light). We correspondingly take the $(d+1)$-dimensional domain $\Gamma=\Omega \times (0,T)$
and consider coarse-graining of functions $f_i\in L^\infty(\Gamma),$ $i=1,2,3,\dots$ with a non-negative, standard mollifier 
$\cG\in C^\infty(\Gamma)$ which can, but need not, be causal.  We assume, for convenience, that ${\rm supp}(\cG)$ 
is contained in the Euclidean unit ball.  Recall that since $L^\infty(\Gamma)\subset L_{loc}^p(\Gamma)$ for $p\geq 1$,  
the functions $f_i$ are locally $p$--integrable, $f_i\in L_{loc}^p(\Gamma)$.  For any open 
$O\subset\subset \Gamma,$ let $\|\cdot\|_{p,O}$ represent the standard $L_p(O)$-norm on the restriction $f_i\big|_O$ . 
All estimates assume $\ell$ sufficiently small for fixed $O\subset\subset \Gamma$, in particular $\ell<\ell_O= {\rm dist}(O,\partial\Gamma)$.   

A basic result is the following:

  \begin{lemma}  \label{tauprop}
For $n>1$, the coarse-graining cumulants are related to cumulants of the difference fields $\delta f(R;X):= f(X+R)- f(X)$ as follows:
\be \tau_\ell(f_1, \dots, f_n) = \langle \delta f_1, \dots, \delta f_n\rangle_\ell^c , \label{cum-avrg} \ee
where $\langle \cdot \rangle_\ell$ denotes average over the displacement vector $R$ with density $\cG_\ell(R)$ and the superscript $c$ 
indicates the cumulant with respect to this average.  
\end{lemma}
This result is proved in \cite{CET1994} for $n=2$ and, in the more general form quoted here, in  \cite{eyinknotes} or \cite{eyinkRec2015}, 
Appendix B. The proof is an easy application of the invariance of cumulants of ``random variables'' to shifts of those variables by ``non-random'' constants.   A direct consequence of Lemma \ref{tauprop} is:
\begin{proposition} 
\emph{(cumulant estimates)} 
\label{bndTau}   For open $O\subset\subset \Gamma,$ $p\in [1,\infty]$ and $n>1$
\be\label{tauIncBnd}
\|\tau_\ell(f_1, \dots, f_n)\|_{p,O} =\cO\left(\prod_{i=1}^n \|\delta f_i(\ell)\|_{p_i,O}\right)  \ \ \ {\rm  with } \ \ \ \frac{1}{p}=\sum_{i=1}^n \frac{1}{p_i} ,
\ee
where $ \|\delta f(\ell)\|_{p,O} :=  \sup_{|R|<\ell}\|\delta f(R)\|_{p,O}$.  Assuming $f_i \in B_{p_i,loc}^{\sigma_i,\infty}(\Gamma)$ with $0<\sigma_i\leq 1$ 
for $i=1, \dots, n$:
\be
\|\tau_\ell(f_1, \dots, f_n)\|_{p,O} =\cO\left(\ell^{\sum_{i=1}^n \sigma_i}\right), 
\ee
If only $f_i \in L^\infty(\Gamma),$ then at least 
\be \lim_{\ell\rightarrow 0} \|\tau_\ell(f_1, \dots, f_n)\|_{p,O} =0, \quad 1\leq p<\infty, \label{density-est}  \ee
but without an estimate of the rate. 
\end{proposition}
Here ``big-$\cO$'' notation, as usual, means inequality up to a constant independent of $\ell$, which in this case depends on the details of the mollifier $\cG$. The final statement is a consequence of the bound (\ref{tauIncBnd}) and the strong continuity of 
the shift operators $(S_{-\br} f)(\bx)=f(\bx+\br)$ in the $L^p(O)$-norm, a standard fact which follows from a simple density argument.  %by a standard density argument, since the formula (\ref{cum-avrg}) implies that (\ref{density-est}) holds when all of the $f_i$ are smooth.   

We also need bounds on space-time derivatives of the cumulants.  This can be accomplished using the fact that all derivatives of cumulants 
with respect to  $X$ can be transferred to space-derivatives of the filter kernels $\cG_\ell(R)$ with respect to $R$. 
This is another consequence of the invariance of cumulants to constant shifts; see \cite{eyinknotes} or 
\cite{eyinkRec2015}. For example, with 
\begin{eqnarray}
\frac{\partial}{\partial X_k} \ol{\tau}_\ell (f_i) &=& \frac{\partial \ol{(f_i)}_\ell}{\partial X_k} =-\frac{1}{\ell} \int \rmd^{d+1} R \ \left(\frac{\partial\cG}{\partial R_k}\right)_\ell(R) \delta f_i(R), 
\label{CGgrad}\\ \nonumber
\frac{\partial}{\partial X_k} \ol{\tau}_\ell (f_i,f_j) &=& -\frac{1}{\ell} \left\{ \int \rmd^{d+1} R \ \left(\frac{\partial\cG}{\partial R_k}\right)_\ell(R)  \delta f_i(R) \delta f_j(R)\right.\\ \nonumber
&&  \! \!\!  -\int \rmd^{d+1} R \ \left(\frac{\partial\cG}{\partial R_k}\right)_\ell(R)  \delta f_i(R) \int \rmd R'\cG_\ell(r') \delta f_j(R')\\
 &&\! \!\!   -\left. \int \rmd^{d+1} R \  \cG_\ell(R) \delta f_i(R) \int \rmd R' \left(\frac{\partial\cG}{\partial R_k'}\right)_\ell(R')  \delta f_j(R') \right\}\!, \label{eq:dtau}
\end{eqnarray}
and so forth. Using expressions of this type, one obtains bounds of the form:
\begin{proposition} 
\emph{(cumulant-derivative  estimates)}
\label{bndTauder} For open $O\subset\subset \Gamma,$ $n\geq 1$ and $\partial_k=\partial/\partial X_k$
\be
\|\partial_{k_1}\cdots \partial_{k_m}\tau_\ell(f_1, \dots, f_n)\|_{p,O} 
=\cO\left(\ell^{-m}\prod_{i=1}^n \|\delta f_i(\ell)\|_{p_i,O}\right)  \ \ \ {\rm  with } \ \ \ \frac{1}{p}=\sum_{i=1}^n \frac{1}{p_i} .
\ee
Assuming $f_i \in B_{p_i,loc}^{\sigma_i,\infty}(\Gamma)$ with $0<\sigma_i\leq 1$ for $i=1, \dots, n$:
\be
\|\partial_{k_1}\cdots \partial_{k_m}\tau_\ell(f_1, \dots, f_n)\|_{p,O} =\cO\left(\ell^{-m+\sum_{i=1}^n \sigma_i}\right) .
\ee
\end{proposition}

For the ``unresolved'' or ``fluctuation'' part of a field $f_\ell':=f-\ol{f}_\ell,$ we have the simple formula
\be  f_\ell'(X)  = -\int \rmd^{d+1} R \ \cG_\ell(R) \delta f(R;X), \ee
which gives
\begin{proposition} \label{lem:fluc}
\emph{(fluctuation estimates)}
\label{bndfpr} For open $O\subset\subset \Gamma$ and $p\in [1,\infty],$  $\|f'_\ell\|_{p,O} 
=\cO\left(\|\delta f(\ell)\|_{p,O}\right)$ and $\|f'_\ell\|_{p,O} =\cO\left(\ell^\sigma\right)$
when also $f\in B_{p,loc}^{\sigma,\infty}(\Gamma)$ for $0<\sigma\leq 1.$
\end{proposition}

Finally, we will also require estimates on $\Delta_\ell h=\ol{h}_\ell-\ul{h}_\ell$ for composite functions of the form $h(f,g)$, where $f,g\in L^\infty(\Gamma)$
and $h$ is a smooth function of two variables. We have the following Lemma:
\begin{lemma}\label{thermobnd}
For  $\ p\geq 1$, let $f\in (B_{p,loc}^{\sigma_p^f,\infty}\cap L^\infty)(\Gamma)$ and $g\in(B_{p,loc}^{\sigma_p^g,\infty}\cap L^\infty)(\Gamma)$.  Let $U\subset \mathbb{R}^2$ be open
and containing the closed convex hull of ${\mathcal R}={\rm ess.ran}(f,g)$, the essential range of the measurable 
function $(f,g)\in L^\infty(\Gamma,\mathbb{R}^2)$.  Consider $H:=h(f,g)$ with $h\in C^1(U,\mathbb{R})$. Then $H\in  
(B_{p,loc}^{\min\{\sigma_p^f, \sigma_p^g\},\infty}\cap L^\infty)(\Gamma).$ 
\end{lemma}
\begin{proof}
Clearly, $H\in L^\infty(\Gamma).$ Since $h\in C^1(U,\mathbb{R})$, the mean value theorem gives:
\begin{eqnarray}
\delta H(R;X) & := &h(f(X+R), g(X+R)) -h(f(X), g(X)) \cr
&=& (\delta f(R;X),  \delta g(R;X)) \cdot\vec{\partial} h(f_*,g_*)
\end{eqnarray}
for $(f_*, g_*)$ on the line segment joining  $(f(X),g(X))$, $(f(X+R),g(X+R))$. We have used the notation  $\vec{\partial}= (\partial/\partial f, \partial/\partial g)$.  
Since ${\mathcal R}\subset U$ is compact, then so also is its closed convex hull ${\rm conv}({\mathcal R})\subset U$ and $\vec{\partial}h$ is bounded 
on ${\rm conv}({\mathcal R})$.  It follows for any open $O\subset\subset \Gamma,$ $|R|<\ell_O,$  $p\geq 1$, 
$\|\delta H(R)\|_{p,O} =\cO\left(|R|^{\min\{\sigma_p^f, \sigma_p^g\}}\right)$.
\hfill $\Box$ \end{proof}
\begin{corollary}\label{prodBes}
Let $f,g$ be as in Lemma \ref{thermobnd}.  Then $f g\in (B_{p,loc}^{\min\{\sigma_p^f,\sigma_p^g\},\infty}\cap L^\infty)(\Gamma)$.
\end{corollary}

The estimate on $\Delta_\ell h=\ol{h}_\ell-\ul{h}_\ell$ is as follows: 
 \begin{proposition}\label{hessThermo}
Let  $h\in C^2(U)$ with $f,g,U$ as in Lemma \ref{thermobnd}.  For  open $O\subset\subset \Gamma$
\be
 \|\Delta_\ell h \|_{p/2,O} =\cO\left(\ell^{2\min\{\sigma_{p}^f,\sigma_{p}^g\}}\right), \quad 
 \mbox{$p\geq 2$} 
\ee
Assuming only $f,g \in L^\infty(\Gamma),$ then at least 
\be \lim_{\ell\rightarrow 0}  \|\Delta_\ell h \|_{p/2,O} =0, \quad 
\mbox{$2\leq p<\infty$,}
\label{density-est-2}\ee
but without an estimate of the rate. 
\end{proposition}
 \begin{proof}
 Using the notation $\vec{\partial}= (\partial/\partial f, \partial/\partial g)$, we have:
\begin{eqnarray*}
 \Delta_\ell h &:=&\   \ol{h(f,g)}_\ell- h(\ol{f}_\ell, \ol{g}_\ell)\\
 &=&\ \Big(\ol{h(f,g)}_\ell- h(f,g) + (f'_\ell,  g'_\ell) \cdot\vec{\partial} h(f,g) \Big) \\
 && +\Big( h(f,g) -h(\ol{f}_\ell, \ol{g}_\ell)-(f'_\ell,  g'_\ell)  \cdot \vec{\partial} h(f,g) \Big).
\end{eqnarray*}
The first term can be rewritten as
\begin{eqnarray*}
  &&\ol{h(f,g)}_\ell- h(f,g) + (f'_\ell,  g'_\ell) \cdot\vec{\partial} h(f,g)  \\
 && = \int d^{d+1}R \ \cG_\ell(R) \ \Big(h(f(X+R),g(X+R))- h(f(X),g(X))\\
 && \quad\quad\quad\quad\quad\quad\quad\quad\quad - (\delta f(R;X),  \delta g(R;X)) \cdot\vec{\partial} h(f(X),g(X)) \Big) \\
  && = \int d^{d+1}R\  \cG_\ell(R) \left.(\vec{\partial}  \vec{\partial})  h\right|_{(f_*,g_*)} : (\delta f(R;X),  \delta g(R;X)) (\delta f(R;X),  \delta g(R;X)),
\end{eqnarray*}
where in the second equality the Taylor theorem with remainder was employed and $(f_*,g_*)$ is defined similarly 
as in Lemma \ref{thermobnd}. Likewise, using $f=\ol{f}_\ell+f'_\ell,$ the second term can be rewritten as 
\begin{eqnarray*}
  && h(f,g) - h(\ol{f}_\ell, \ol{g}_\ell) - (f'_\ell,  g'_\ell) \cdot\vec{\partial} h(f,g)  \\
 && \qquad = \left.(\vec{\partial}  \vec{\partial})  h\right|_{(f_{\star},g_{\star})} : (f'_\ell, g'_\ell) ( f'_\ell, g'_\ell),
\end{eqnarray*}  
and $(f_\star,g_\star)$ is a point on the line segment connecting $(\ol{f}_\ell(X),\ol{g}_\ell(X))$,$(f(X),g(X)).$
Note that $(\ol{f}_\ell(X),\ol{g}_\ell(X))\in {\rm conv}({\mathcal R})$ because the coarse-grained field with a non-negative 
mollifier $\cG_\ell$ is a limit of averages of values in ${\rm ess.ran.}(f,g).$ Thus, 
$ \left.(\vec{\partial}  \vec{\partial})  h\right|_{(f_{\star},g_{\star})}$ is uniformly bounded, since 
$(\vec{\partial}  \vec{\partial})  h$ is bounded on ${\rm conv}({\mathcal R}).$ It follows from the above 
formulas, the H\"older inequality, and Proposition \ref{lem:fluc} that 
\be
\|\Delta_\ell h\|_{p/2,O} =\cO\left( \max\{\|\delta f(\ell)\|_{p,O}, \|\delta g(\ell)\|_{p,O}\}^2\right). 
\label{Lpbd-comp} \ee
The above estimate immediately yields $\|\Delta_\ell h\|_{p/2,O} =\cO\left(\ell^{2\min\{\sigma_{p}^f,\sigma_{p}^g\}}\right)$
assuming the appropriate Besov regularity.

The final statement of the proposition is obtained from the estimate (\ref{Lpbd-comp}) and the strong continuity of 
the shift operators in the $L^p(O)$-norm.   \hfill $\Box$ \end{proof}

%Since $f_\delta'=f-\ol{f}_\delta$ can be made arbitrarily small in the $L_{p/2}$-norm for any 
%$f\in L^{p/2}(\Gamma)$ with $\delta$ sufficiently small, one can write
%\be \ol{h(f,g)}_\ell = \ol{h(\ol{f}_\delta,\ol{g}_\delta)}_\ell + \ol{\vec{\partial}s(f_{\delta *},g_{\delta *})\cdot(f'_\delta,g'_\delta)}\ee 
%where $(f_{\delta *},g_{\delta *})$ is on the line-segment between $(f,g)$ and $(\ol{f}_\delta,\ol{g}_\delta)$, and 
%\be h( \ol{f}_\ell,\ol{g}_\ell)= h(\ol{(\ol{f}_\delta)}_\ell,\ol{(\ol{g}_\delta)}_\ell)
%+ \vec{\partial}s(f_{\delta\star},g_{\delta\star})\cdot \ol{(f'_\delta,g'_\delta)}_\ell\ee 
%for $(f_{\delta\star},g_{\delta\star})$ on the line-segment between $(\ol{f}_\ell,\ol{g}_\ell)$ and $\ol{(\ol{f}_\delta)}_\ell,\ol{(\ol{g}_\delta)}_\ell)$. 
%The correction terms in the two equations above can be made arbitrarily small in $L_{p/2}$-norm,
%while for $\delta$ fixed $\lim_{\ell\rightarrow 0}
%\|\ol{h(\ol{f}_\delta,\ol{g}_\delta)}_\ell - h(\ol{(\ol{f}_\delta)}_\ell,\ol{(\ol{g}_\delta)}_\ell)\|_{p/2}
%\rightarrow 0$ by (\ref{Lpbd-comp}). 

One last estimate will be needed:

\begin{proposition}\label{derivThermo}
Let  $h\in C^1(U)$ with $f,g,U$ as in Lemma \ref{thermobnd}.  For  open $O\subset\subset \Gamma$
\be
 \|\nabla_x \ul{h}_\ell \|_{p,O} =\cO\left(\ell^{\min\{\sigma_{p}^f,\sigma_{p}^g\}-1}\right), \quad p\geq 1.
\ee
\end{proposition} 
\begin{proof}
By the chain rule, $\nabla_x \ul{h}=\vec{\partial}h(\ol{f}_\ell,\ol{g}_\ell)\cdot (\nabla_x\ol{f}_\ell,\nabla_x\ol{g}_\ell).$
Since $(\ol{f}_\ell,\ol{g}_\ell)$ is in the closed convex hull of ${\mathcal R},$ one immediately obtains from 
Proposition \ref{bndTauder} that 
\be \|\nabla_x\ul{h}_\ell\|_{p,O} =\cO\left( \frac{1}{\ell}\max\{\|\delta f(\ell)\|_{p,O},\|\delta g(\ell)\|_{p,O}\}\right), \label{grad-ulh} \ee
which gives the claimed estimate for the assumed Besov regularity. 
\hfill $\Box$ \end{proof}

%%%%%%%%%%%%%%%%%%%%%%%%%%%%%%%%%%%%%%%%%%%%%%%%%%%%%%%%%%%%%%%%%%%%%%%%%%%%%%%%%%%%%%%%%%%%%%%%%%%%%%%%%%%%%%%%%%%%%%%%%%%%%%%%%%%%%%%%%%%%%%%%%%%%%%%%%%%%%%%%%%%%%%%%%%%%%%%%%%%%%%%%%%%%%%%%%%%%%%%%%%%%%%%%%%%%%%%%%%%%%%%%%%%%%%%%%%%%%%%%%%%%%%%%%%%%%%%%%%%%%%%%%%%%%%%%%%%%%%%%%%%%%%%%%%%%%%

\section{Proof of Theorem \ref{Onsager1E}}\label{sec:proofThm1}

By assumption $u,\vr,\bu\in L^\infty(\Omega \times (0,T))\subset L_{loc}^p(\Omega \times (0,T))$. 
We shall obtain estimates in $L^p(O)$ for any open set $O\subset\subset \Gamma$. To simplify 
expressions in the proof, we let $O$ be implicit in this section and everywhere use $\|\cdot\|_p$  
to denote the $L_p(O)$-norm $\|\cdot\|_{p,O}$ .  Also, all estimates assume $\ell<\ell_O= {\rm dist}(O,\partial\Gamma)$.
We consider in order the three balance equations (\ref{eq:KE*Euler})--(\ref{eq:s*Euler}) in Theorem 1.

\vspace{2mm}

\noindent \emph{Kinetic Energy:}  Setting $\ve= 0,$ the coarse-grained kinetic energy 
balance (\ref{eq:CGKE}) for compressible Navier-Stokes simplifies, because terms involving $\bT^\ve$ vanish:
\be
\partial_t \left( \frac{1}{2} \ol{\vr}_\ell |\tilde{\bu}_\ell |^2\right) + \nabla_x \cdot \bJ_\ell^v =    \ol{p}_\ell\ol{\Theta}_\ell  -Q_\ell^{\rm flux}  ,
\label{inert-eq:CGKE2}
\ee
where the various terms are defined by:
 \begin{eqnarray}  \label{inert-eq:Juell}
\bJ_\ell^v &:=& \left(\frac{1}{2}\ol{\vr}_\ell|\tilde{\bu}_\ell|^2 
+\ol{p}_\ell\right)\tilde{\bu}_\ell+ \ol{\vr}_\ell\tilde{\bu}_\ell\cdot\tilde{ \tau}_\ell (\bu, \bu)  -\frac{\ol{p}_\ell }{\ol{\vr}_\ell }\ol{ \tau}_\ell(\vr, \bu) , \\
 \label{inert-eq:Quell}
Q_\ell^{\rm flux} &:=& \frac{\nabla_x\ol{p}_\ell }{\ol{\vr}_\ell } \cdot \ol{ \tau}_\ell(\vr, \bu)- \ol{\vr}_\ell \nabla_x \tilde{\bu}_\ell :\tilde{ \tau}_\ell(\bu,\bu) .
 \end{eqnarray}
We now consider the limit as $\ell\to 0$ of the equation (\ref{inert-eq:CGKE2}). 
Of course, by standard results, $\ol{u}_\ell,$ $\ol{\vr}_\ell,$ $\ol{\bu}_\ell,$ $\ol{p}_\ell$ $\rightarrow $ $u,$ $\vr,$ $\bu,$ $p$ 
strong in $L_{loc}^p$ for any $1\leq p<\infty$ (see e.g. \cite{gilbarg2015elliptic}, Lemma 7.2 or 
\cite{evans2015measure}, \S 4.2.1, Theorem 1). As a special case of (\ref{favref})
\be \tilde{\bu}_\ell=\ol{\bu}_\ell+\ol{\tau}_\ell(\vr,\bu)/\ol{\vr}_\ell, \label{favrev} \ee 
which implies for any $p\geq 1$ that 
$$ \|\tilde{\bu}_\ell-\bu\|_p \leq \|\ol{\bu}_\ell-\bu\|_p + \|1/\vr\|_\infty\|\ol{\tau}_\ell(\vr,\bu)\|_p,$$
so that $\tilde{\bu}_\ell\rightarrow \bu$ strongly as well. Here (\ref{density-est}) of Proposition \ref{bndTau} was used. 
We infer that $ \frac{1}{2} \ol{\vr}_\ell |\tilde{\bu}_\ell|^2$ converges to $\frac{1}{2} \vr |\bu|^2$ strong in $L_{loc}^p$
for any $p\geq 1,$ and thus 
\be \partial_t \left( \frac{1}{2} \ol{\vr}_\ell |\tilde{\bu}_\ell |^2\right) \Dto \ \partial_t \left( \frac{1}{2} \vr |\bu|^2\right) \ee
as $\ell\to 0. $
Using the special case of (\ref{favreTau})
\be \tilde{\tau}_\ell(\bu,\bu) = \ol{\tau}_\ell(\bu,\bu) + \frac{1}{\ol{\vr}_\ell}\ol{\tau}_\ell(\vr,\bu,\bu) 
-\frac{1}{\ol{\vr}^2_\ell}\ol{\tau}_\ell(\vr,\bu)\ol{\tau}_\ell(\vr,\bu), \label{favreStress} \ee
one obtains by exactly similar arguments with Proposition \ref{bndTau} that 
\be \nabla_x \cdot \bJ_\ell^v\ \Dto\ \nabla_x  \left( (\frac{1}{2} \vr |\bu|^2+p)\bu\right). \ee
Also, under our assumptions, $Q_\ell^{\rm flux}$ has a distributional limit:
\be Q_\ell^{\rm flux}\ \Dto\ Q_{\rm flux}. \ee
Thus, all of the terms in (\ref{inert-eq:CGKE2}) except 
$\ol{p}_\ell\ol{\Theta}_\ell$ have been proved to have distributional limits as $\ell\to 0.$ It follows that the limit of 
$\ol{p}_\ell\ol{\Theta}_\ell$ also exists and equals $-Q_{\rm flux}- \partial_t \left( \frac{1}{2} \vr |\bu|^2\right)-\nabla_x  
\left( (\frac{1}{2} \vr |\bu|^2+p)\bu\right),$ independent of choice of $\cG$. 
Thus, 
\be \ol{p}_\ell\ol{\Theta}_\ell\ \Dto\ p \circ \Theta \ee
which completes the derivation of the kinetic energy balance (\ref{eq:KE*Euler}).
\vspace{4mm}

\noindent \emph{Internal Energy}: From (\ref{eq:KE*Euler}), the internal energy constructed as $u=E- \frac{1}{2}\vr |\bu|^2$, satisfies (\ref{eq:IE*Euler}) distributionally.  This could be alternatively deduced by considering the $\ell \to 0$ limit of the intrinsic resolved internal energy balance (\ref{eq:CGIE*})  with $\ve=0$.
\vspace{2mm}

\noindent \emph{Entropy}: Setting $\ve=0$ in the intrinsic resolve entropy equation (\ref{eq:LSs*}), we obtain
\be
\partial_t \ul{s}_\ell^*+ \nabla_x \cdot \bJ_\ell^{s*} =   \Sigma_\ell^{\rm inert *} ,  \label{E-eq:LSs*}
\ee
for 
 \begin{eqnarray}  \label{E-eq:Jsell*}
\bJ_\ell^{s*} &:=& \bJ_\ell^s + \ul{\beta}_\ell \bJ_\ell^k,\\
 \label{E-eq:Jsell}
\bJ_\ell^s &:=& \ul{s}_\ell \ol{\bu}_\ell+ \ul{\beta}_\ell {\ol{\tau}_\ell (u,\bu)} -\ul{\lambda}_\ell\ol{ \tau}_\ell(\vr, \bu),\\
\label{E-Jkell} 
 \bJ_\ell^k &:=&\frac{1}{2} \ol{\vr}_\ell \tilde{\tau}_\ell(v_i,v_i)\tilde{\bu}_\ell + \ol{\tau}_\ell (p,\bu)+ \frac{1}{2} \ol{\vr}_\ell \tilde{\tau}_\ell(v_i,v_i,\bu) ,
\end{eqnarray}
and, with $\Sigma_\ell^{\rm inert *}=  - I_\ell^{\rm flux} +  \Sigma_\ell^{\rm flux *},$ for 
 \begin{eqnarray}  \label{E-eq:Isell}
I_\ell^{\rm flux}&:=& \ul{\beta}_\ell (\ol{p}_\ell -\ul{p}_\ell) \ol{\Theta}_\ell,\\
  \label{E-eq:Lambell*}
 \Sigma_\ell^{\rm flux *}&:=&  \Sigma_\ell^{\rm flux} +  \ul{\beta}_\ell Q_\ell^{\rm flux} + \partial_t\ul{\beta}_\ell\ k_\ell + \nabla_x\ul{\beta}_\ell \cdot \bJ_\ell^k  ,\\
 \label{E-eq:Lambell}
 \Sigma_\ell^{\rm flux}&:=&  \nabla_x \ul{\beta}_\ell \cdot\ol{\tau}_\ell (u,\bu) - \nabla_x \ul{\lambda}_\ell \cdot\ol{ \tau}_\ell(\vr,\bu). 
 \end{eqnarray}
We next show that  $\partial_t \ul{s}_\ell^*+ \nabla_x \cdot \bJ_\ell^{s*}\Dto\ \partial_t s + \nabla_x \cdot (s\bu)$ as $\ell\to 0.$
Note that  
$$\|s(\ol{u}_\ell,\ol{\vr}_\ell)-s(u,\vr)\|_p\leq \|  \ol{s(u,\vr)}_\ell-s(u,\vr)\|_p+\|  \ol{s(u,\vr)}_\ell-s(\ol{u}_\ell,\ol{\vr}_\ell)\|_p.$$ 
Obviously $\ol{s}_\ell\to s$ strong in $L_{loc}^p$ for $p\geq 1,$ but also $\|\Delta_\ell s\|_p\rightarrow 0$ by (\ref{density-est-2}) of Proposition 
\ref{hessThermo}.  
Thus,  $\ul{s}_\ell\to s$ strong in $L_{loc}^p$. Also, $\|\ul{\beta}_\ell k_\ell\|_p\rightarrow 0$ by (\ref{density-est}) of Proposition \ref{bndTau}. 
It follows that $\ol{s}_\ell^*\to s$ strong in $L_{loc}^p$  for $p\geq 1$ and thus 
$$ \partial_t \ul{s}_\ell^*\ \Dto \ \partial_t s(u,\vr).$$ 
Using the formula (\ref{favreStress}) for $\tilde{\tau}_\ell(\bv,\bv)$ and the similar formula for $\tilde{\tau}_\ell(\bv,\bv,\bv)$
that follows from (\ref{favreTau3}), then similar arguments with Propositions \ref{bndTau} and \ref{hessThermo} show that 
$\bJ_\ell^{s*} \Dto s\bu$ strong in $L_{loc}^p$  for $p\geq 1$ and thus 
$$ \nabla_x \cdot \bJ_\ell^{s*} \ \Dto \ \nabla_x\cdot \left( s(u,\vr)\bu\right). $$
We infer from (\ref{E-eq:LSs*}) that the distributional limit of $\Sigma_\ell^{\rm inert *}$ as $\ell\to 0$ exists and  
is equal to 
$\Sigma_{\rm flux}:=\partial_t s + \nabla_x \cdot (s\bu).$ 
Thus, entropy balance (\ref{eq:s*Euler}) holds, with  
\be \Sigma_\ell^{\rm inert *}\ \Dto\ \Sigma_{\rm flux}. \ee
This completes the proof of Theorem 1. \hfill $\Box$
\vspace{2mm}

\section{Proof of Theorem \ref{Onsager1NS}}\label{sec:proofThm1NS}

%\noindent \emph{\textbf{Proof of Theorem \ref{Onsager1NS}}}\\

\vspace{-2mm}
To prove that the strong limits of $u^\ve,$ $\vr^\ve,$ $\bu^\ve$ in $L_{loc}^p(\Gamma)$ for some $1\leq p<\infty$ 
as $\ve\rightarrow 0$ satisfy the Euler equations weakly, we use the concept of ``coarse-grained solution'' 
discussed in section \ref{sec:CGWS}. The coarse-grained Navier-Stokes system with 
transport coefficients scaled by $\ve$ appears the same as (\ref{eq:CGrho})--(\ref{eq:CGTE})
except that there is now a factor $\ve$ implicitly contained in the terms $\bT^\ve$ and $\bq^\ve$ wherever 
they appear.  Our strategy shall be to show that, pointwise in space-time, these terms indeed vanish as $\ve\rightarrow 0,$
while all of the other terms in the coarse-grained Navier-Stokes equation converge pointwise as $\ve\rightarrow 0$ 
to the corresponding terms in the coarse-grained Euler equations for the limiting fields $u,$ $\vr,$ $\bu.$

Here again, we let the open set $O\subset\subset \Gamma$ be implicit in the estimates below 
and use $\|\cdot\|_p$ to represent the $L_p(O)$-norm. We also assume that  $\ell<\ell_O= {\rm dist}(O,\partial\Gamma)$.
We first note that the properties that (i) $\| f^\ve\|_\infty$ is bounded uniformly in $\ve$ and (ii) 
$f^\ve\rightarrow f$ in $L_{loc}^p(\Gamma)$ for $1\leq p<\infty$ as 
$\ve\rightarrow 0$ for the basic fields $f^\ve=u^\ve,$ $\vr^\ve,$ $\bu^\ve$ immediately implies that the 
same is true for simple product functions such as $\bj^\ve=\vr^\ve\bu^\ve,$ $\vr^\ve |\bu^\ve|^2,$
$\vr^\ve |\bu^\ve|^2\bu^\ve$, etc. For compositions $h^\ve:=h(u^\ve,\vr^\ve)$ with thermodynamic functions
such as $h=T,$ $p,$ $\mu$, $\eta,$ $\zeta,$ $\kappa$ we need the precise Assumption 
\ref{smoothAss} on smoothness of $h$ with $M=1$. Of course, ${\mathcal R}^\ve,{\mathcal R}\subset K$ 
for $\ve<\ve_0,$ so that $\|h^\ve\|_\infty$ is bounded uniformly for $\ve<\ve_0$ and $\|h\|_\infty$ satisfies
the same bound. Furthermore, we can write
\begin{eqnarray}
&& h(u^\ve(X),\vr^\ve(X))-h(u(X),\vr(X)) \cr
&& \hspace{40pt} =  \vec{\partial} h(u_*,\vr_*)\cdot (u^\ve(X)-u(X),\vr^\ve(X)-\vr(X)),
\end{eqnarray}  
where $(u_*,\vr_*)$ is on the line segment between $(u^\ve(X),\vr^\ve(X))$ and $(u(X),\vr(X))$. 
Since $(u_*,\vr_*)\in K,$ then, by Assumption \ref{smoothAss}, the 2-vector $\ell_q$-norm $|\vec{\partial} h(u_*,\vr_*)|_q$ 
with $q=p/(p-1)$ is bounded 
by the maximum value $C_{h,q}$ of $|\vec{\partial} h|_q$ on $K.$ It thus follows 
easily that 
\be  \|h(u^\ve,\vr^\ve)-h(u,\vr)\|_p\leq C_{h,q}[\| u^\ve-u\|_p^p+ \|\vr^\ve-\vr\|_p^p]^{1/p}, \ee 
so that $h^\ve=h(u^\ve,\vr^\ve)$ also satisfies $\|h^\ve-h\|_p\rightarrow 0$ for the same $p$ as $\ve\rightarrow 0.$   
Thus $h^\ve\rightarrow h$ in $L_{loc}^p(\Gamma)$. Next note from the identity (\ref{CGgrad}) that 
\be
\frac{\partial}{\partial X_k} \ol{(f^\ve-f)}_\ell(X) 
  = -\frac{1}{\ell} \int \rmd^{d+1} R \ \left(\frac{\partial\cG}{\partial R_k}\right)_\ell(R-X) (f^\ve(R)-f(R)), 
\ee 
Hence, for each $X,$ 
\be |\partial_k \ol{(f^\ve-f)}_\ell(X)|\leq (c_{\ell,p}/\ell)\|f^\ve-f\|_p \ee
with $c_{\ell,p}=\|(\partial\cG)_\ell\|_q$ for $q=p/(p-1)$ and thus $\partial_k \ol{(f^\ve)}_\ell(X)\rightarrow \partial_k \ol{f}_\ell$
as $\ve\rightarrow 0$ whenever $f^\ve\rightarrow f$ in $L_{loc}^p(\Gamma)$. Applying this result with $f=\vr,$ $\bj,$ $\bj \bu,$
$p,$ $E,$ $(E+p)\bu,$ we get that pointwise in space-time 
 \begin{eqnarray} \label{convCGrho}
\partial_t \ol{\vr^\ve}_\ell +\nabla_x\cdot\ol{\nj}^\ve_\ell &\longrightarrow& \partial_t \ol{\vr}_\ell + \nabla_x\cdot\ol{\,\nj\,}_\ell  ,\\
 \label{convCGj}
\partial_t \ol{\,\nj\,}^\ve_\ell + \nabla_x \cdot \left( \ol{(\bj^\ve  \bu^\ve)}_\ell +  \ol{p}^\ve_\ell \bI \right) & \longrightarrow &
\partial_t \ol{\,\nj\,}_\ell + \nabla_x \cdot \left( \ol{(\bj  \bu)}_\ell +  \ol{p}_\ell \bI \right) ,\\
\label{conCGTE}
\partial_t \ol{E}_\ell^\ve + \nabla_x \cdot \left( \ol{((E^\ve+p^\ve)\bu^\ve)}_\ell\right) &\longrightarrow &
\partial_t \ol{E}_\ell + \nabla_x \cdot \left( \ol{((E+p)\bu)}_\ell\right) ,
\end{eqnarray}
as $\ve\rightarrow 0.$ The coarse-grained Euler equations 
 \begin{eqnarray} \label{E-eq:CGrho}
\partial_t \ol{\vr}_\ell &+& \nabla_x\cdot\ol{\,\nj\,}_\ell=0  ,\\
 \label{E-eq:CGj}
\partial_t \ol{\,\nj\,}_\ell &+& \nabla_x \cdot \left( \ol{(\bj  \bu)}_\ell +  \ol{p}_\ell \bI \right) =\bzed,\\
\label{E-eq:CGTE}
\partial_t \ol{E}_\ell &+& \nabla_x \cdot \left( \ol{((E+p)\bu)}_\ell\right) =0,
\end{eqnarray}
follow for $u,$ $\vr,$ $\bu$ if $\nabla_x\cdot\ol{(\bT^\ve)}_\ell,$ $\nabla_x\cdot\ol{(\bT^\ve\cdot\bu^\ve)}_\ell,$
and $\nabla_x\cdot\ol{(\bq^\ve)}_\ell$ all vanish as $\ve\rightarrow 0.$ 

We first consider the shear-viscosity contribution to $\nabla\cdot\ol{(\bT^\ve)}_\ell.$
With the shorthand notation $\eta^\ve(X) :=\ve \eta(u^\ve(X),\vr^\ve(X)),$ we can bound 
this using Cauchy-Schwartz  inequality as 
\begin{eqnarray} 
&& \left|\nabla_x\cdot \overline{(2\eta^\ve\bS^\ve)}_\ell(X)\right| = \frac{2}{\ell}\left|\int \rmd^{d+1}R\ (\nabla_x \cG)_\ell(R) \cdot 
\eta^\ve(X+R) \bS^\ve(X+R)\right| \cr 
&& \qquad \leq  \frac{2}{\ell} \sqrt{\int_{{\rm supp(\cG_\ell)}} \rmd^{d+1}R\  \eta^\ve(X+R) 
\times \int |(\partial \cG)_\ell(R-X)|^2 \ Q_\eta^\ve(dR)},\cr
&& \label{CS-etaS} \end{eqnarray} 
with $Q_\eta^\ve(\rmd R) = 2\eta^\ve(R) |\bS(R)|^2 \rmd^{d+1}R$ denoting the kinetic-energy dissipation 
measure for $\ve>0.$ Finally, because $Q_\zeta^\ve\geq 0,$
\be \left|\nabla_x\cdot \overline{(2\eta^\ve\bS^\ve)}_\ell(X)\right|
\leq  \frac{2}{\ell} \sqrt{\int_{{\rm supp(\cG_\ell)}}\!\! \!\!\!\rmd^{d+1}R\  \eta^\ve(X+R) 
\times\!\! \int |(\partial \cG)_\ell(R-X)|^2 \ Q^\ve(dR)} \ee
with $Q^\ve=Q_\eta^\ve+Q_\zeta^\ve.$ Since $\cG_\ell\in D(\Gamma)$ implies that $S_{X}|\partial \cG_\ell|^2\in D(\Gamma)$ 
also whenever ${\rm dist}(X,\partial\Gamma)<\ell$, then 
\be
\lim_{\ve\rightarrow 0}\int  |(\partial \cG)_\ell(R-X)|^2 \, Q^\ve(dR)
=  \int  |(\partial \cG)_\ell(R-X)|^2 \, Q(dR) 
\ee 
by Assumption \ref{assum1b}.  On the other hand, because $\eta(u^\ve,\vr^\ve)\in L^\infty(\Gamma)$
when $\eta$ satisfies the smoothness Assumption \ref{smoothAss} with $M=0,$ then the upper
bound in (\ref{CS-etaS}) is proportional to $\ve^{1/2}.$ Thus, $\nabla_x\cdot \overline{(2\eta^\ve\bS^\ve)}_\ell(X)
\to 0$ as $\ve\to 0$ for $\ell>{\rm dist}(X,\partial\Gamma)$. An identical argument using $Q_\eta^\ve\geq 0$ 
shows that likewise $\nabla_x\overline{(\zeta^\ve\Theta^\ve)}_\ell(X)\to 0$
as $\ve\to 0,$ and both results together imply that $\nabla\cdot\ol{(\bT^\ve)}_\ell\to 0$ pointwise. 

In a similar manner, the shear-viscosity contribution to $\nabla_x\cdot\ol{(\bT^\ve\cdot\bu^\ve)}_\ell$ can be 
bounded as 
\begin{eqnarray} 
&& \left|\nabla_x\cdot \overline{(2\eta^\ve\bS^\ve\cdot\bu^\ve)}_\ell(X)\right|\cr
&& \qquad \qquad= \frac{2}{\ell}\left|\int \rmd^{d+1}R\ (\nabla_x \cG)_\ell(R) \cdot 
\eta^\ve(X+R) \bS^\ve(X+R)\cdot\bu^\ve(X+R)\right| \cr 
&& \qquad \qquad\leq  \frac{2}{\ell} \sqrt{\int_{{\rm supp(\cG_\ell)}} \rmd^{d+1}R\  
   \eta^\ve(X+R)|\bu^\ve(X+R)|^2}\cr
&&    \hspace{150pt} \times \sqrt{\int |(\partial G)_\ell(R-X)|^2 \ Q^\ve(dR)},\cr
&& \label{CS-etaSv} \end{eqnarray} 
and an analogous bound holds for $\nabla_x\cdot \overline{(2\zeta^\ve\Theta^\ve\bu^\ve)}_\ell.$ Thus, by 
Assumption \ref{assum1b} $\nabla_x\cdot\ol{(\bT^\ve\cdot\bu^\ve)}_\ell\to 0$ pointwise as $\ve\to 0.$

Finally, $\nabla_x\cdot\ol{(\bq^\ve)}_\ell=-\nabla\cdot\ol{(\kappa^\ve\nabla_xT^\ve)}_\ell$ and the 
entropy-production measure due to thermal conductivity is defined by $\Sigma_\kappa^\ve(\rmd R)=
\kappa^\ve(R) \left|  \frac{\nabla_x T^\ve(R)}{T^\ve(R)}\right|^2\rmd^{d+1}R$ for $\ve>0.$ 
Because $Q^\ve/T^\ve\geq 0,$ thus $\Sigma_\kappa^\ve\leq \Sigma^\ve.$  Writing 
$\kappa^\ve\nabla_xT^\ve= \sqrt{\kappa^\ve}  T^\ve \cdot \sqrt{\kappa^\ve}\frac{\nabla_xT^\ve}{T^\ve}$ and using 
a Cauchy-Schwartz estimate similar to (\ref{CS-etaSv}), it follows from the convergence 
$\Sigma^\ve\Dto \Sigma$ in Assumption \ref{assum1b} that $\nabla_x\cdot\ol{(\bq^\ve)}_\ell\to 0$
pointwise as $\ve\to 0$ for $\ell>{\rm dist}(X,\partial\Gamma)$.

In conclusion, the coarse-grained Euler equations (\ref{E-eq:CGrho})--(\ref{E-eq:CGTE}) 
hold for all $X$ with ${\rm dist}(X,\partial\Gamma)<\ell$ and for all $\ell>0.$
By Proposition \ref{CGequiv} in section \ref{sec:CGWS}, we have thus proved that $(u,\vr,\bu)$ form a weak Euler solution. 
As an aside, we note that it would clearly suffice for this statement to have in Assumption \ref{assum1b} 
only the condition on entropy-production $\Sigma^\ve\Dto \Sigma$ and not the additional 
assumption $Q^\ve\ \Dto \ Q$.  If in Theorem \ref{Onsager1NS} only the statement (\ref{eq:s*}) on entropy balance were made, 
then this would be more economical in terms of hypotheses. However, to derive the balance equations (\ref{eq:KE*}) and (\ref{eq:IE*}) 
we need the additional convergence statement in Assumption \ref{assum1b} for $Q^\ve$ as we now show.  

To derive the balance equations of kinetic energy, internal energy and entropy for the weak Euler solutions, 
we start with the corresponding eqs.(\ref{eq:KE}),(\ref{eq:IE}),(\ref{eq:s}) for compressible Navier-Stokes. 
Then, because the basic fields $u^\ve,$ $\vr^\ve,$ $\bu^\ve$ and 
their compositions with functions $h^\ve:=h(u^\ve,\vr^\ve)$ satisfying the smoothness assumptions 
converge strongly in $L_{loc}^p$ for some $1\leq p<\infty$ to the corresponding  fields $u,$ $\vr,$ $\bu$ and $h(u,\vr),$ it follows directly that 
\begin{eqnarray} \label{lim-eq:KE}
\nonumber
 \partial_t \left(\frac{1}{2}\vr^\ve |\bu^\ve|^2\right) + \nabla_x \cdot \bigg(\bigg(p^\ve \!\! &+&\! \!\frac{1}{2}\vr^\ve |\bu^\ve|^2 \bigg)\bu^\ve\bigg)  \\
\nonumber
&\Dto &  \partial_t \left(\frac{1}{2}\vr |\bu|^2\right) + \nabla_x \cdot \left(\left(p+\frac{1}{2}\vr |\bu|^2 \right)\bu\right) , \\  
\label{eps-eq:IE}
\nonumber
\partial_t  u^\ve + \nabla_x \cdot \left( u^\ve\bu^\ve\right) &\Dto &  \partial_t  u + \nabla_x \cdot (u\bu)  ,\\ 
\label{lim-eq:s}
\partial_t s^\ve+ \nabla_x \cdot \left(s^\ve \bu^\ve\right)  &\Dto& \partial_t s+ \nabla_x \cdot (s\bu) .
\end{eqnarray}
To see that 
$$ \nabla_x\cdot\left(\bT^\ve\cdot \bu^\ve \right), \ \nabla_x\cdot \bq^\ve, \ 
\nabla_x \cdot \left(\frac{\bq^\ve}{T^\ve} \right) \Dto 0, $$
note that this is equivalent to $\nabla_x\ol{(\bT^\ve\cdot \bu^\ve)}_\ell,$ $\nabla_x\ol{\bq^\ve}_\ell,$
$\ol{(\bq^\ve/T^\ve)}_\ell\to 0$ pointwise. This has already been proved for the first two, and is shown 
for the third by a very similar Cauchy-Schwartz argument by writing $\bq^\ve/T^\ve=-\sqrt{\kappa^\ve}\cdot \sqrt{\kappa^\ve}\nabla_x T^\ve/T^\ve.$

Because of the condition $\Sigma^\ve\Dto\Sigma$ in Assumption \ref{assum1b}, all of the terms in the Navier-Stokes
entropy balance (\ref{eq:s}) converge distributionally and thus one obtains in the limit $\ve\to 0$ the entropy
balance (\ref{eq:s*}) for the weak Euler solution. Similarly, because of the condition $Q^\ve\ \Dto \ Q$ in Assumption 
\ref{assum1b}, all of the terms in the Navier-Stokes kinetic energy and internal energy balances 
(\ref{eq:KE}),(\ref{eq:IE}) are proved to converge distributionally, except $p^\ve \Theta^\ve$.  Thus, this term 
also converges 
\begin{eqnarray*}
 \Dlim_{\ve\rightarrow 0} p^\ve \Theta^\ve &= &
 \partial_t \left(\frac{1}{2}\vr |\bu|^2\right) + \nabla_x \cdot \left(\left(p+\frac{1}{2}\vr |\bu|^2\right) \bu \right) +Q\cr
 &=& Q- [\partial_t  u + \nabla_x \cdot \left( u\bu\right) ].
\end{eqnarray*} 
With the notation 
$ p*\Theta :=\Dlim_{\ve\rightarrow 0} p^\ve \Theta^\ve $
we thus obtain the balances (\ref{eq:KE*}),(\ref{eq:IE*}) of kinetic and internal energy for the limiting weak 
Euler solution.  \hfill $\Box$

%%%%%%%%%%%%%%%%%%%%%%%%%%%%%%%%%%%%%%%%%%%%%%%%%%%%%%%%%%%%%%%%%%%%%%%%%%%%%%%%%%%%%%%%%%%%%%%%%%%%%%%%%%%%%%%%%%%%%%%%%%%%%%%%%%%%%%%%%%%%%%%%%%%%%%%%%%%%%%%%%%%%%%%%%%%%%%%%%%%%%%%%%%%%%%%%%%%%%%%%%%%%%%%%%%%%%%%%%%%%%%%%%%%%%%%%%%%%%%%%%%%%%%%%%%%%%%%%%%%%%%%%%%%%%%%%%%%%%%%%%%%%%%%%%%%%%%

 \section{Proof of Theorem \ref{Onsager2}}\label{sec:proofThm2}

The strategy to prove Theorem 2 is to use the commutator estimates developed in Section \ref{sec:proofs} to show that 
$Q_{\rm flux}$ and $\Sigma_{\rm flux}$ vanish when the Euler 
solutions possess suitable Besov regularity.   Then, we use the ``inertial-range'' expressions (\ref{compressible45ths}) to show the dissipation 
measures $Q$ and $\Sigma$ also vanish, and that $p*\Theta = p\circ \Theta$. 
We again make implicit the open set $O\subset\subset \Gamma$, let $\|\cdot\|_p$ represent the $L_p(O)$-norm, 
and assume that $\ell<\ell_O= {\rm dist}(O,\partial\Gamma)$.
\vspace{2mm}

\noindent \emph{Energy Flux:} We first show that $Q_{\rm flux}$ defined by (\ref{QcascadeDefects}),(\ref{eq:Quell}) necessarily exists and vanishes for weak Euler solutions satisfying the exponent inequalities (\ref{sigineq1})--(\ref{sigineq3}). To show this, simple bounds can be derived for $Q_\ell^{\rm flux}$ 
using the expressions (\ref{favrev}), (\ref{favreStress}) and Propositions \ref{bndTau} and \ref{bndTauder}. 
One obtains 
$$\|(1/\ol{\vr}_\ell)\nabla_x\ol{p}_\ell \cdot \ol{ \tau}_\ell(\vr, \bu)\|_{p/3} =\cO\left( \|1/\vr\|_\infty \frac{1}{\ell}\|\delta p(\ell)\|_p
\|\delta\vr(\ell)\|_p\|\delta\bu(\ell)\|_p\right), \quad\  p\geq 3,$$
\vspace{-4mm}
$$\|\nabla_x \tilde{\bu}_\ell\|_{p}=\frac{1}{\ell}\|\delta\bu(\ell)\|_p \left[ \cO(1) + \cO(\|1/\vr\|_\infty\|\vr\|_\infty)
+ \cO(\|1/\vr\|_\infty^2\|\vr\|_\infty^2)\right], \quad \  p\geq 1 ,$$
$$\|\tilde{ \tau}_\ell(\bu,\bu)\|_{p/2}=\|\delta\bu(\ell)\|_p^2 \left[ \cO(1) + \cO(\|1/\vr\|_\infty\|\vr\|_\infty)
+ \cO(\|1/\vr\|_\infty^2\|\vr\|_\infty^2)\right], \  p\geq 2, $$
and thus 
\be \|Q_\ell^{\rm flux}\|_{p/3} = \cO\left(\frac{1}{\ell}\|\delta p(\ell)\|_p
\|\delta\vr(\ell)\|_p\|\delta\bu(\ell)\|_p\right)+ \cO\left( \frac{\|\delta\bu(\ell)\|_p^3 }{\ell}  \right), \quad\quad\quad\  p\geq 3. \label{Qflux-est} \ee
In this latter estimate we absorb the dependence upon the maximum-to-minimum mass ratio $\|1/\vr\|_\infty\|\vr\|_\infty$ 
into the constant factor, since this ratio is $\ell$-independent. Assuming the Besov regularity of $u,$ $\vr,$ $\bu$ 
in Theorem \ref{Onsager2} and using Lemma \ref{thermobnd} to get the Besov regularity of $p$, one thus obtains 
$$ \|Q_\ell^{\rm flux}\|_{p/3} = \cO\left(\ell^{\min\{\sigma_p^u,\sigma_p^\vr\}+\sigma_p^\vr+\sigma_p^v-1}\right)+ \cO\left(\ell^{3\sigma_p^v-1}\right),
\quad p\geq 3. $$
It follows that 
$$ 2\min\{\sigma_p^u,\sigma_p^\vr\}+\sigma_p^v>1, \ 3\sigma_p^v>1, \mbox{ for some } \ p\geq 3
\Longrightarrow  \Dlim_{\ell\to 0} Q_\ell^{\rm flux}=0. $$

This is enough to infer the first statement of Theorem \ref{Onsager2} that $Q_{\rm flux}$ exists and vanishes for weak Euler solutions, but not enough to conclude that the viscous anomaly vanishes, $Q=0.$   
Recall by (\ref{compressible45ths}) that
\be Q = Q_{\rm flux} + \tau(p,\Theta).
 \ee 
 Therefore, with the exponent inequalities assumed above, we can only conclude
\be Q=\tau(p,\Theta) := p*\Theta - p\circ \Theta. \label{Q-taupTheta} \ee
%Indeed, it follows from $Q_\ell^{\rm flux}\Dto 0$ and (\ref{QeqQinert}) that 
%\be Q=\Dlim_{\ell\rightarrow 0} \tau_\ell(p,\Theta) = p*\Theta - \Dlim_{\ell\rightarrow 0} \ol{p}_\ell\ol{\Theta}_\ell, \ee 
%so that the latter distributional limit exists and equals $p*\Theta-Q$ independent of the mollifier $\cG.$
% Note that in the definition (\ref{Dprod}) of $p\circ \Theta$ it is enough to consider non-negative mollifiers $\cG\geq 0$ \cite{oberguggenberger1986products,oberguggenberger1986distributions}. 
In order to show that $Q=0,$ we must make use of the entropy balance, which we consider next. 
\vspace{2mm}

\noindent \emph{ Entropy Anomaly}:
We show that $\Sigma_{\rm flux}$ defined by (\ref{cascadeDefects}) necessarily exists and vanishes for weak Euler solutions satisfying the exponent inequalities (\ref{sigineq1})--(\ref{sigineq3}).   To accomplish this, we next derive bounds on $\Sigma_\ell^{\rm inert *}$ using (\ref{E-eq:Isell})--(\ref{E-eq:Lambell})
and Propositions \ref{bndTau}, \ref{bndTauder}, \ref{hessThermo}, and \ref{derivThermo}. 
Expression (\ref{E-eq:Isell}) and Propositions \ref{bndTauder}, \ref{hessThermo} give:
$$ \|I_\ell^{\rm flux}\|_{p/3}=\cO\left( \frac{1}{\ell}\max\{\|\delta u(\ell)\|_p,\|\delta\vr(\ell)\|_p\}^2\|\delta\bu(\ell)\|_p\right).$$
Expression (\ref{E-eq:Lambell}) and Propositions \ref{bndTau}, \ref{derivThermo} give:
\begin{eqnarray}
\|\Sigma_\ell^{\rm flux}\|_{p/3}&=&\cO\left(\|\nabla_x\ul{\beta}_\ell\|_p
\|\delta u(\ell)\|_p\|\delta\bu(\ell)\|_p\right)+\cO\left(\|\nabla_x \ul{\lambda}_\ell\|_p 
\|\delta\vr(\ell)\|_p\|\delta\bu(\ell)\|_p\right)\cr 
&=& \cO\left( \frac{1}{\ell}\max\{\|\delta u(\ell)\|_p,\|\delta\vr(\ell)\|_p\}^2\|\delta\bu(\ell)\|_p\right),
\end{eqnarray}
while Propositions \ref{bndTau}, \ref{derivThermo} give for the added terms to $\Sigma_\ell^{\rm flux *}$ in 
(\ref{E-eq:Lambell*}) the estimates 
\begin{eqnarray}\nonumber
\|\partial_t \ul{\beta}_\ell k_\ell\|_{p/3} &=& \cO\left(\|\partial_t\ul{\beta}_\ell\|_p\|\delta\bu(\ell)\|_p^2\right)
=\cO\left( \frac{1}{\ell}\max\{\|\delta u(\ell)\|_p,\|\delta\vr(\ell)\|_p\}\|\delta\bu(\ell)\|_p^2\right),\\\nonumber
\| \nabla_x\ul{\beta}_\ell \cdot \bJ_\ell^k \|_{p/3}&=&\cO\left(\|\nabla_x\ul{\beta}_\ell\|_p\|\delta\bu(\ell)\|_p^2\right)
=\cO\left( \frac{1}{\ell}\max\{\|\delta u(\ell)\|_p,\|\delta\vr(\ell)\|_p\}\|\delta\bu(\ell)\|_p^2\right).
\end{eqnarray}
To estimate $k_\ell$ and $\bJ_\ell^k$ we here used the expressions (\ref{favrev}) for $\tilde{\bu}_\ell,$
(\ref{favreStress}) for $\tilde{\tau}_\ell(\bu,\bu)$ and the similar expression for $\tilde{\tau}_\ell(\bu,\bu,\bu)$
that follows from (\ref{favreTau3}). Assuming the Besov regularity of $u,$ $\vr,$ $\bu$ 
in Theorem \ref{Onsager2}, one thus obtains from these estimates and the estimate of $\ul{\beta}_\ell Q_\ell^{flux}$
using (\ref{Qflux-est}) that for any $p\geq 3$
$$ \|\Sigma_\ell^{\rm inert *}\|_{p/3}= 
\cO\left(\ell^{2\min\{\sigma_p^u,\sigma_p^\vr\}+\sigma_p^v-1}\right)
+\cO\left(\ell^{\min\{\sigma_p^u,\sigma_p^\vr\}+2\sigma_p^v-1}\right)
+\cO\left(\ell^{3\sigma_p^v-1}\right). $$
The inequalities (\ref{sigineq1})--(\ref{sigineq3}) thus imply that $\Sigma_\ell^{\rm inert *}\to 0$ strong in $L_{loc}^{p/3}$
as $\ell\to 0$ for the same choice of $p\geq 3.$ Because of (\ref{compressible45ths}), 
it follows that the non-ideal entropy production also vanishes $\Sigma\equiv 0.$
\vspace{2mm}

\noindent \emph{Viscous Energy Dissipation Anomaly:}  We now show that $\Sigma=0$ implies that $Q=0.$ First note 
$$
\Sigma^\ve\geq \beta^\ve  Q^\varepsilon \geq  Q^\ve/ \|T^\ve\|_\infty.
$$
Because $\|T^\ve\|_\infty$ by Assumption \ref{assum1} is bounded by some constant $T_0$ uniformly in $\ve<\ve_0,$
we thus find that 
$$
\Sigma^\ve\geq Q^\ve/T_0 \geq 0, \quad \ve<\ve_0, 
$$
and one obtains in the limit $\ve\to 0$  that 
$$ 0= \Sigma\geq Q/T_0 \geq 0. $$
Thus, the inequalities (\ref{sigineq1})--(\ref{sigineq3}) in Theorem \ref{Onsager2} for some $p\geq 3$
imply also $Q\equiv 0$. 
\vspace{2mm}

\noindent \emph{Pressure-Dilatation Defect:}  
Lastly, the result $Q=\tau(p,\Theta)$ in (\ref{Q-taupTheta}) together with $Q\equiv 0$ implies that $p*\Theta=p\circ\Theta,$ as was claimed.  \hfill $\Box$

%%%%%%%%%%%%%%%%%%%%%%%%%%%%%%%%%%%%%%%%%%%%%%%%%%%%%%%%%%%%%%%%%%%%%%%%%%%%%%%%%%%%%%%%%%%%%%%%%%%%%%%%%%%%%%%%%%%%%%%%%%%%%%%%%%%%%%%%%%%%%%%%%%%%%%%%%%%%%%%%%%%%%%%%%%%%%%%%%%%%%%%%%%%%%%%%%%%%%%%%%%%%%%%%%%%%%%%%%%%%%%%%%%%%%%%%%%%%%%%%%%%%%%%%%%%%%%%%%%%%%%%%%%%%%%%%%%%%%%%%%%%%%

\section{Proof of Theorem \ref{timeReg}}\label{sec:proofThm3}

We derive Theorem \ref{timeReg} from a result for more general balance equations (\ref{gen-bal}). We consider 
cases where $\bv\in L^\infty(\Omega \times (0,T);\mathbb{R}^m),$ so that ${\mathcal R}={\rm ess.ran.}(\bv)$
is a compact subset of $\mathbb{R}^m$ with $K={\rm conv}({\mathcal R})$ also compact, and $\bF=\bF(\bv)$ is 
a $C^1$ function on an open set $U,$ $K\subset U\subset \mathbb{R}^m.$ Furthermore, the individual
components of $F_{ia}$ of $\bF$ for $i=1,\dots,d$ and $a=1,\dots,m$ may not depend upon all of the components
$u_a,$ $a=1,\dots,m$ of $\bv$ but only upon a subset. We assume that for each $a=1,\dots,m$  the $d$-vector
$\bF_a=(F_{1a},\dots,F_{da})$ is a function of the form
\be \bF_a(\bv)= \tilde{\bF}_a(u_{b^{(a)}_1},\dots,u_{b^{(a)}_{m_a}}), \quad a=1,\dots, m \label{dep-ass} \ee 
where the subset $\mathbb{M}_a=\{b^{(a)}_1,\dots,b^{(a)}_{m_a}\}\subset \{1,\dots,m\}$ has cardinality $m_a\leq m,$ and thus $\bF_a$ is constant in the 
variables $u_{b}$ for $b\notin \mathbb{M}_a$

We then have the following general result:

\begin{customthm}{4*}\label{GentimeReg}
Suppose that $\bv\in L^\infty(\Omega \times (0,T);\mathbb{R}^m)$ is a weak solution of (\ref{gen-bal}) where 
$\bF\in C^1(U)$ with $U$ open and ${\rm conv}({\rm ess.ran.}(\bv))\subset U\subset \mathbb{R}^m,$ and that 
also $\bF_a$ satisfies the condition (\ref{dep-ass}) for each $a=1,\dots,m.$ If for some $p\geq 1$
\be\label{regassFi}
 u_a\in L^\infty((0,T); B_{p,loc}^{\sigma_p^a,\infty}(\Omega)), \quad 0< \sigma_p^a\leq 1; \qquad a=1,\dots,m,
\ee
where the above spaces are defined by (\ref{Linfty_into_BesovLoc}), then 
\be\label{fspacetimeReg}
 u_a\in B_{p,loc}^{\bar{\sigma}_p^a ,\infty}(\Omega \times (0,T)), \quad \bar{\sigma}_p^a= \min\{\sigma_p^a,\min_{b\in \mathbb{M}_a}\sigma_p^{b}\};
 \qquad a=1,\dots,m. 
\ee
\end{customthm}

\begin{proof}
We use the notation $\Gamma=\Omega \times (0,T)$ and $R=(\br,\tau)\in \Gamma.$
Since $L^\infty(\Gamma)\subset L_{loc}^p(\Gamma)$ and
$p\geq 1,$ we must only bound the requisite $L^p(O)$-norm in the definition (\ref{besov-def}) of the local 
space-time Besov norm for any open $O\subset\subset \Gamma$. For $R=(\br,\tau)$ with $|R|<R_O={\rm dist}(O,\partial\Gamma),$
Minkowski's inequality gives:
\be\label{splitSpaceTime}
\| u_a(\cdot+R) - u_a\|_{L^p(O)}\leq \| u_a(\cdot, \cdot+\tau ) - u_a\|_{L^p( O')}+  \| u_a(\cdot+\br,\cdot) - u_a \|_{L^p(O)}
\ee
where $O'=S_\br O := \{(\bx+{\br},t) :  (\bx,t)\in O\}\subset\subset\Gamma$. 
The assumed uniform regularity (\ref{regassFi})  guarantees that $\| u_a(\cdot+\br,\cdot) - u_a \|_{L^p(O)}=\cO(|\br|^{\sigma_p^a})$. 
To estimate the time-increment term, fix an $0<\ell\leq |\tau|$ and decompose $\bv= \hat{\bv}_\ell+\bv'_\ell$ 
with $\hat{\bv}_\ell=\bv*\check{G}_\ell$ for a spatial mollifier $G_\ell$. Applying Minkowski's inequality again,
\begin{eqnarray}\nonumber
\| u_a(\cdot, \cdot+\tau ) - u_a\|_{L^p(O')} &\leq& \| \hat{u}_{a,\ell}(\cdot, \cdot+\tau )  - \hat{u}_{a,\ell}\|_{L^p(O')}\\
&& \hspace{10mm} +  \| u_{a,\ell}'(\cdot, \cdot+\tau ) - u_{a,\ell}'\|_{L^p(O')}.  \label{timeIncBnd} 
\end{eqnarray}
In order to estimate these terms, it is convenient to assume that $O=O_r\times O_t,$ a space-time product of open sets, 
and thus $O'=O_r'\times O_t'$ as well.  It clearly suffices to consider product sets, because any other pre-compact 
open set can be strictly included in such a product set.  Since $\partial_t u_a+\nabla_x\cdot\bF_a=0$ is satisfied in the sense 
of distributions or, equivalently, pointwise after space-time mollification (see Proposition \ref{CGequiv}),  standard approximation 
arguments show:
  \begin{eqnarray}\nonumber
 \| \hat{u}_{a,\ell}(\cdot, \cdot+\tau ) - \hat{u}_{a,\ell}\|_{L^p(O'_r\times O'_t)} &\leq& |\tau|   \|\nabla_x \cdot\hat{ \bF}_{a,\ell}\|_{L^\infty(O'_t ; L^p(O'_r))}\\ \nonumber
%&\leq& |\tau|   \|\nabla_x \cdot\hat{ \bF}_{a,\ell}\|_{L^\infty((0,T) ; L^p(O'_r))}\\
 &=& \cO(\ell^{\mu_p^a -1}|\tau| ),  \quad \mu_p^a=\min_{b\in\mathbb{M}_a}\sigma_p^b. \label{timeIncBnd2}
 \end{eqnarray}
Here we have used the inherited spatial Besov regularity of ${\bold F_a}$ with exponent $\mu_p^a$, which follows 
from a straightforward generalization of Lemma \ref{thermobnd}, and the spatial version of 
Proposition \ref{bndTauder}. 
On the other hand, the term involving the fluctuation fields can be bounded using the spatial analogue of Proposition \ref{lem:fluc} as:
\be\label{timeIncBnd3}
\| u_{a,\ell}'(\cdot, \cdot+\tau ) - u_{a,\ell}'\|_{L^p(O'_r\times O'_t))}\leq 2\|u_{a,\ell}'\|_{L^\infty(O_t'), L^p(O'_r)) } =  \cO(\ell^{\sigma_p^a}).
 \ee
 From equations (\ref{timeIncBnd})--(\ref{timeIncBnd3}) we obtain
  \begin{eqnarray}\label{BesovTimeIncBnd}
\| u_a(\cdot, \cdot+\tau ) - u_a\|_{L^p(O')} =\cO(\ell^{\mu_p^a -1}|\tau|)   +   \cO(\ell^{\sigma_p^a}).
\end{eqnarray}
Since $\ell\leq |\tau|<1$ by assumption, we increase the upper bound in (\ref{BesovTimeIncBnd}) by replacing both $\mu_p^a$ 
and $\sigma_p^a$ with their minimum, $\bar{\sigma}_p^a$, in (\ref{fspacetimeReg}).
The resulting bound is then optimized by choosing the arbitrary scale $\ell\leq  |\tau|$ to be $\ell \propto |\tau|$.  
Altogether,
  \begin{eqnarray}
  \label{timeAndSpaceIncBnd}
\| u_a(\cdot, \cdot+\tau ) - u_a\|_{L^p(O')} &=&\cO(|\tau|^{\bar{\sigma}_p^a} ), \\
  \| u_a(\cdot+\br,\cdot) - u_a \|_{L^p(O)}&=&\cO(|\br|^{\bar{\sigma}_p^a}).  \label{timeAndSpaceIncBnd2}
\end{eqnarray}
It follows from (\ref{splitSpaceTime}) and (\ref{timeAndSpaceIncBnd}),(\ref{timeAndSpaceIncBnd2}) that  $u_a\in B_{p,loc}^{\bar{\sigma}_p^a,\infty}(\Omega \times (0,T)).$
\hfill $\Box$ 
\end{proof}

\begin{proof}[Theorem \ref{timeReg}]
The result is proved as a corollary of Theorem \ref{GentimeReg},  
%\ref{GentimeReg}
specialized to the compressible Euler system with $(u_0,u_1,\dots,u_d,u_{d+1}):=(\vr,j_1,\dots,j_d, E)$ and 
  \begin{eqnarray*}
  F_{i,0} &:=&  u_i, \\
  F_{i,j} &:=&   u_0^{-1} u_i u_j+ p(u, u_0)\delta_{ij}, \\
  F_{i,d+1} &:=&  \left( u_{d+1}+p(u,u_0)\right)u_0^{-1}u_i.
\end{eqnarray*}
for $i,j=1,\dots,d$ and $u:= u_{d+1}- \frac{u_1^2+\cdots+u_d^2}{2u_0}$. The assumed strict positivity of $\vr\geq\vr_0>0$, space-time boundedness 
of $\bv,$ and smoothness of $p$ implies that $ \bF$ possesses the requisite regularity. It follows that:
  \begin{eqnarray*}
\vr \in B_{p,loc}^{\min\{\sigma_p^{\vr},\sigma_p^{j}\},\infty}(\Omega \times (0,T)), \quad \bj, E\in B_{p,loc}^{\min\{\sigma_p^{\vr},\sigma_p^{j},\sigma_p^{E}\},\infty}(\Omega \times (0,T)), 
  \end{eqnarray*}
  Recalling that the fields $\bj$ and $E$ are algebraically related to $u,$ $\vr$, $\bu$ by $\bj :=\vr \bu$ and  $E:=\frac{1}{2}\vr |\bu|^2+u$, an application of Corollary \ref{prodBes} shows that we may take $\sigma_p^{j}= \min\{\sigma_p^{\vr},\sigma_p^{v}\}$ and $\sigma_p^{E}= \min\{\sigma_p^{u},\sigma_p^{\vr},\sigma_p^{v}\}$.  
 The inverse relations $\bu=\vr^{-1}\bj$ and $u=E-\vr^{-1}|\bj|^2$ and another application of Corollary \ref{prodBes} yields 
 the space-time regularity (\ref{thm3-1})--(\ref{thm3-3}) claimed in Theorem 3. 
\hfill $\Box$ \end{proof}

\begin{remark}
Theorem \ref{GentimeReg} applies also to solutions of the incompressible Euler equations with velocity $\bu$ and (kinematic) 
pressure $P$ satisfying $\bu,P\in L^\infty(\Gamma),$  for $\Gamma=\mathbb{T}^d\times (0,T).$
Assuming for $q\geq 1$ that $\bu\in L^\infty((0,T),B_q^{\sigma_q,\infty}(\mathbb{T}^d))$, 
elliptic regularization of the solutions of the Poisson equation 
$$ -\triangle P = \partial^2(v_iv_j)/\partial x_i\partial x_j$$
implies that  $P\in L^\infty((0,T),B_{q}^{\sigma_{q},\infty}(\mathbb{T}^d))$.  Alternatively, this regularity of $P$ follows
from boundedness of Calder\'on-Zygmund operators in Besov-space norms. 
Theorem \ref{GentimeReg} yields $\bu\in B_q^{\sigma_q,\infty}(\mathbb{T}^d \times (0,T)),$  so that $\bu$ is as regular in time as it is in space. 
\end{remark}

%%%%%%%%%%%%%%%%%%%%%%%%%%%%%%%%%%%%%%%%%%%%%%%%%%%%%%%%%%%%%%%%%%%%%%%%%%%%%%%%%%%%%%%%%%%%%%%%%%%%%%%%%%%%%%%%%%%%%%%%%%%%%%%%%%%%%%%%%%%%%%%%%%%%%%%%%%%%%%%%%%%%%%%%%%%%%%%%%%%%%%%%%%%%%%%%%%%%%%%%%%%%%%%%%%%%%%%%%%%%%%%%%%%%%%%%%%%%%%%%%%%%%%%%%%%%%%%%%%%%%%%%%%%%%%%%%%%%%%%%%%%%%%%%%%%%%%

\begin{acknowledgements} 
We would like to thank Hussein Aluie for sharing his unpublished work. T.D. would like to thank Daniel Ginsberg for useful discussions.  Research of T.D. is partially supported by NSF-DMS grant 1703997 and a Fink Award from the Department of Applied Mathematics \& Statistics at the Johns Hopkins University.
\end{acknowledgements}

\bibliography{bibliography}
\end{document}